\documentclass[12pt,a4paper,reqno]{amsart}
\usepackage[latin1]{inputenc}
\usepackage[english]{babel}
\usepackage{amsmath}
\usepackage{amsfonts}
\usepackage{amssymb}
\usepackage{mathrsfs}
\usepackage{latexsym}

\usepackage{mathrsfs}

\usepackage[margin=3cm]{geometry}

\usepackage{color}

\usepackage{graphicx}

\usepackage[plainpages=false,colorlinks,hyperindex,bookmarksopen,linkcolor=red,citecolor=blue,urlcolor=blue]{hyperref}

\newtheorem{theorem}{Theorem}[section]
\newtheorem{definition}{Definition}

\newtheorem{prop}{Proposition}[section]
\newtheorem{remark}{Remark}[section]

\numberwithin{equation}{section}

\allowdisplaybreaks

\author{Mirko D'Ovidio}
        \address{Department of Basic and Applied Sciences for Engineering\newline
Sapienza University of Rome\newline via A. Scarpa 10, Rome, Italy}
        \email[Corresponding author]{mirko.dovidio@uniroma1.it}

\title[Probability and non-local Boundary Value Problems]{On the non-local boundary value problem from the probabilistic viewpoint}
\begin{document}

\maketitle

\begin{abstract}
We provide a short introduction of new and well-known facts relating non-local operators and irregular domains. Cauchy problems and boundary value problems are considered in case non-local operators are involved. Such problems respectively lead to anomalous behavior on the bulk and on the surface of a given domain. Such a behavior can be considered (in a macroscopic viewpoint) in order to describe regular motion on irregular domains (in the microscopic viewpoint).  
\end{abstract}

\tableofcontents

\section{Introduction}

We present a collection of results about the connection between non-local operators and irregular domains. We are concerned with the interplay between macroscopic and microscopic analysis of Feller motions on bounded and unbounded domains. Our discussion relies on  the fact that anomalous motions on regular domains can be considered in order to describe motions on irregular domains. Here, by irregular domains we mean domains with irregular boundaries or interfaces, for example of fractal type.

Fractional Cauchy problems have been considered as models for motions on non-homogeneous domains, fractals for instance. In the same spirit we consider anomalous motions for non-local boundary value problems, that is the motions exhibit anomalous behavior only near the boundary. Non-local boundary conditions therefore introduce models for motions on irregular domains, for example domains with trap boundaries (trap domains, in the sequel). 

The presentation will run as follows. First we recall some facts about non-local operators and the associated processes. Sometimes non-local operators are also referred as to generalized fractional operators. However, there is no fractional order to be considered. We move to the next sections by discussing the non-local initial value problems and then the non-local boundary value problems. For the latter, we underline some connection with the case of random Koch domains and, in particular, with the boundary behaviors introduced by the non-local boundary conditions.

\section{Markov processes, time changes, non-local operators}

\subsection{Markov Processes}
Let $X=\{X_t\}_{t\geq 0}$ with $X_0=x \in \Omega$ be the Markov process on $\Omega$ with generator $(G,D(G))$ where $(G,D(G))$ is the generator of the semigroup $Q_t$, then we have the probabilistic representation
\begin{align*}
Q_t f(x) = \mathbf{E}_x[f(X_t)].
\end{align*}
We recall that for the Markovian semigroup $Q_t$, the generator is defined by 
\begin{align}
\label{limitGenerator}
G f(x) = \lim_{t\downarrow 0} \frac{Q_t f(x) - f(x) }{t}, \quad f \in D(G)
\end{align}
for $f$ in some set of real-valued functions and for which the limit exists in some sense. Then we say that $f$ belongs to the domain of the generator $D(G)$.  
In case $X$ is a Feller process for instance, the domain $D(G)$ consists of continuous functions vanishing at infinity for which \eqref{limitGenerator} exists as uniform limit. Examples of Feller processes are given by the Brownian motion (with $G f= f^{\prime \prime}$), the Poisson process (with $Gf = \lambda (F - 1) f$ where $Ff(x)=f(x+1)$ is the forward operator) and in general, L\'{e}vy processes (for which we will consider $G = -\Psi(-\Delta)$ to be better defined further on). For the Brownian motion the forward and backward Kolmogorov's equation have the same reading in terms of heat equations. A Poisson process (started at zero for instance) is driven by $\lambda (1-B)f$ where $Bf(x)=f(x-1)$ is the backward operator. We skip here a discussion on L\'{e}vy processes.

We say that $X$ is a Feller process if $X$ is a strong Markov process, right-continuous with no discontinuity other than jumps. If $X$ is a continuous strong Markov process we say that $X$ is a Feller diffusion (or simply, diffusion).

For the process $X$ started at $x$ with density $p(t,x,y)$ it holds that
\begin{align*}
\mathbf{P}_x(X_t \in D) = \int_D p(t,x,y)dy, \quad D \subseteq \Omega
\end{align*}
and 
\begin{align*}
\mathbf{E}_x[f(X_t)] = \int_{\Omega} f(y) p(t,x,y)dy, \quad t>0, \; x \in \Omega
\end{align*} 
where $\mathbf{E}_x$ is the expectation with respect to $\mathbf{P}_x$. The function $u(t,x)=Q_t f(x)$ solves the Cauchy problem
\begin{align*}
\frac{\partial u}{\partial t} = G u, \quad u_0 = f \in D(G).
\end{align*}

\subsection{Subordinators and Inverses}
The Bernstein function 
\begin{equation}
\Phi(\lambda) = \int_0^\infty \left( 1 - e^{ - \lambda z} \right) \phi(dz), \quad \lambda \geq 0 \label{LevKinFormula}
\end{equation} 
where $\phi$ on $(0, \infty)$ with $\int_0^\infty (1 \wedge z) \phi(dz) < \infty$ is now a L\'evy measure can be associated with a L\'evy process. Thus, the symbol $\Phi$ is the Laplace exponent of a subordinator $H_t$ with $H_0=0$, that is 
$$\mathbf{E}_0[\exp( - \lambda H_t)] = \exp(- t \Phi(\lambda)).$$ 
We also recall that
\begin{align}
\label{tailSymb}
\frac{\Phi(\lambda)}{\lambda} = \int_0^\infty e^{-\lambda z} \overline{\phi}(z)dz, \qquad \overline{\phi}(z) = \phi((z, \infty))
\end{align}
and $\overline{\phi}$ is the so called \emph{tail of the L\'evy measure}.
 
Let us introduce the inverse process 
$$L_t := \inf \{s \geq 0\,:\, H_s \geq t \} = \inf \{s \geq 0\,:\, H_s \notin (0,t) \}$$
with $L_0=0$. We do not consider (except in some well mentioned case) step-processes with $\phi((0, \infty)) < \infty$ and therefore we focus only on strictly increasing subordinators with infinite measures. Thus, the inverse process $L$ turns out to be a continuous process. Both random times $H$ and $L$ are not decreasing. By definition, we also can write
\begin{align}
\label{relationHL}
\mathbf{P}_0(H_t < s) = \mathbf{P}_0(L_s>t), \quad s,t>0.
\end{align}
It is worth noting that $H$ can be regarded as an hitting time for a Markov process whereas, $L$ can be regarded as a local time for a Markov process (\cite{BerBook}). We denote by $h$ and $l$ the following densities 
\begin{align*}
\mathbf{P}_0(H_t \in dx) = h(t,x)dx, \quad \mathbf{P}_0(L_t \in dx) = l(t,x)dx.
\end{align*}
As usual we denote by $\mathbf{P}_0$ the probability of a process started at $x=0$. 

\begin{remark}
For $\Phi(\lambda)=\lambda^\alpha$ (that is, for stable subordinators), it holds the following relation between densities
\begin{align*}
\frac{h(v,z)}{l(z,v)} = \alpha \frac{v}{z}, \quad z,v>0.
\end{align*}
This result has been stated in the famous book \cite{BerBook} without proof, thus we refer to \cite{DovSPL}.  
\end{remark}

From \eqref{relationHL}, straightforward calculation gives 
\begin{align}
\label{Lapdensityl}
\int_0^\infty e^{-\lambda t} l(t,x)\, dt = \frac{\Phi(\lambda)}{\lambda} e^{-x\, \Phi(\lambda)}, \quad \lambda>0.
\end{align}
It suffices to consider that 
$$l(t,x) = - \frac{d}{dx}\mathbf{P}_0(L_t > x) = - \frac{d}{dx}\mathbf{P}_0(t > H_x)$$
in \eqref{Lapdensityl}. We recall that (\cite{CapDovFCAA}), for a good function $f$, for $\lambda> \Phi^{-1}(w)$,
\begin{equation}
\mathbf{E}_0\left[ \int_0^\infty e^{-\lambda t} f(L_t) dt \right] =\frac{\Phi(\lambda)}{\lambda} \mathbf{E}_0 \left[ \int_0^\infty e^{- \lambda H_t} f(t)\, dt  \right].
\end{equation}

Denote by $\widetilde{l}(\lambda, x)$ the Laplace transform \eqref{Lapdensityl}. Assume that $\forall\, x>0$, for $n\geq 1$,
\begin{align}
\label{assumel}
\lim_{\lambda \to 0^+ } \lambda^n \, \widetilde{l}(\lambda, x) = 0 \quad \textrm{and} \quad \lim_{\lambda \to \infty} \lambda^n \, \widetilde{l}(\lambda, x) = 0.
\end{align}
Then, we conclude that $\forall\, x>0$, for $n\geq 1, $ $\lambda^n \, \widetilde{l}(\lambda, x) \in C_0([0, \infty))$ and in particular, 
\begin{align*}
\exists M>0\, :\, \forall\, x>0, \; | \lambda^n \, \widetilde{l}(\lambda, x) | \leq M.
\end{align*}
Thus, under \eqref{assumel},
\begin{align*}
\forall\, x >0 \quad l(\cdot, x) \in C^\infty((0, \infty)).
\end{align*}
A further check shows that ($n=0$ and formula \eqref{Lapdensityl})
\begin{align*}
\forall\, x>0 \quad l(\cdot, x) \in L_1((0, \infty)) \quad \textrm{iff} \quad \lim_{\lambda \downarrow 0} \frac{\Phi(\lambda)}{\lambda} <\infty .
\end{align*}

Concerning $\widetilde{h}(t, \xi) = e^{-t \Phi(\xi)}$ assume that $\forall\, t>0$, for $n\geq 1$, $\xi^n\, \widetilde{h}(t, \xi) \in C_0([0, \infty))$ and therefore, $\xi^n\, \widetilde{h}(t, \xi)$ is bounded as well. Then, we have that
\begin{align*}
\forall\, t>0 \quad h(t, \cdot) \in C^\infty((0, \infty)).
\end{align*}
Here we obviously have that
\begin{align*}
\forall\, t>0 \quad h(t, \cdot) \in L_1((0, \infty)).
\end{align*}

\begin{remark}
\label{rmk:LTuniqueness}
We notice that the Laplace transform $\widehat{u}(\lambda)$ with $\lambda>\omega$ exists for $u$ in the set of piecewise continuous functions of exponential order $w$. That is, given $M>0$, $\exists\, T$ such that $|u(t)|\leq M e^{t\omega}$ for $t>T$. This fact will be also discussed further on.
\end{remark}

\begin{prop}
We have that $l(z,0)=\overline{\phi}(z)$, $z>0$.
\end{prop}
\begin{proof}
Just considering \eqref{Lapdensityl}.
\end{proof}

\begin{prop}
Let us write
\begin{align*}
\kappa(z) = \int_0^\infty h(t,z)dt, \quad \textrm{and} \quad \ell(z)= l(z,0), \quad z>0. 
\end{align*}
Then, $\kappa$ and $\ell$ are associated Sonine kernels. Moreover,
\begin{align}
\label{conSonineK}
\int_\mathbb{R} \kappa(z)\, \ell(1-z)\, dz = \int_0^1 \kappa(z)\, \ell(1-z)\, dz = 1.
\end{align}
\end{prop}
\begin{proof}
It is enough to consider the Laplace transforms of $\kappa$ and $\ell$.
\end{proof}


We notice that, for a stable subordinator, for which $\Phi(z)=z^\alpha$,
\begin{align*}
\mathbf{E}_0 \left[ \int_0^\infty e^{-\lambda H_t} dt \right] = \frac{1}{\lambda^\alpha} = \int_0^\infty e^{-\lambda x} \frac{x^{\alpha-1}}{\Gamma(\alpha)} dx
\end{align*}
from which we obtain the potential density of $H_t$, that is
\begin{align*}
\int_0^\infty h(t,x)dt = \frac{x^{\alpha-1}}{\Gamma(\alpha)}=: \kappa(x), \quad x >0.
\end{align*}
The inverse of $\lambda^{\alpha-1}$ gives
\begin{align*}
l(t,0) = \frac{t^{-\alpha}}{\Gamma(1-\alpha)} =: \ell(t), \quad t>0.
\end{align*}
Thus, formula \eqref{conSonineK} is verified. The convolutions $\kappa *u$ and $\ell *u$ can be used in order to define respectively the fractional integral and the fractional derivative of a given function $u$. Indeed, 
\begin{align*}
\mathcal{D}^\alpha_x u(x) := \frac{d}{dx} \mathcal{I}^{1-\alpha}_x u(x) = \frac{1}{\Gamma(1-\alpha)} \frac{d}{dx} \int_0^x u(y) (x-y)^{-\alpha} dy  
\end{align*}
is the well-known Riemann-Liouville derivative written in terms of the associated fractional integral. Notice that
\begin{align*}
\mathcal{D}^\alpha_x u(x) = \frac{d}{dx} (u * \ell)(x), \quad \mathcal{I}^\alpha_x u(x) = (u * \kappa)(x).
\end{align*}
We also recall the Caputo-Dzherbashian derivative
\begin{align*}
D^\alpha_x u(x) := \frac{1}{\Gamma(1-\alpha)} \int_0^x u^\prime(s)(x-s)^{-\alpha} ds = \mathcal{D}^\alpha_x \big(u(x) - u(0) \big) 
\end{align*}
where $u^\prime = du/ds$. The Caputo-Dzherbashian derivative (also termed Caputo derivative) has been introduced by Caputo in \cite{caputoBook, CapMai71a, CapMai71b} and previously by Dzherbashian in \cite{Dzh66,DzhNers68}. The Riemann-Liouville derivatives have a long history. \\

The previous operators are obtained by convolution and represent well-known objects in fractional calculus where $\alpha$ has a clear meaning in terms of fractional order. Some criticisms have been also reported in the literature about the definition of derivative and integral for some operators (see for example \cite{OrtTMac}). In the following, we consider non-local operators for which we do not have a direct link with a fractional order but with a symbol $\Phi$ and therefore, with a convolution kernel (singular or not, see the last section). We basically move from the following "types" of operators:
\begin{itemize}
\item[i)] Marchaud (type) operators,
\begin{align*}
\mathbf{D}^{\Phi}_{x-} u(x) = \int_0^\infty \big( u(x) - u(x-y) \big)  \phi(dy)
\end{align*}
and
\begin{align*}
\mathbf{D}^{\Phi}_{x+} u(x) = \int_0^\infty \big( u(x) - u(x+y) \big)  \phi(dy);
\end{align*} 
\item[ii)] Riemann-Liouville (type) operators,
\begin{align*}
\mathcal{D}^\Phi_{x-} u(x) = \frac{d}{dx} \int_{-\infty}^x u(x-y) \overline{\phi}(y)dy
\end{align*}
and
\begin{align*}
\mathcal{D}^\Phi_{x+} u(x) = \frac{d}{dx} \int^{\infty}_x u(x-y) \overline{\phi}(y)dy;
\end{align*}
\item[iii)] Caputo-Dzherbashian (type) operators,
\begin{align*}
\mathfrak{D}^\Phi_t u(t) = \int_0^t u^\prime(t-s) \overline{\phi}(y)dy 
\end{align*}
where $u^\prime=du/dt$ is the derivative of $u(t)$.
\end{itemize}
These operators coincide in a certain class of functions. A key fact is played by the extension with zero for the negative part of the real line. That is, for those functions $u$ for which $u(x)=0$ for $x\leq 0$, a quick check shows that, under the right assumptions for existence of Laplace transforms in each definition, the Laplace symbols coincide.

\subsection{Non-local (space) operators}
\label{spaceNLop}

Let $(G, D(G))$ be the generator of the semigroup $Q_t$. We consider the representation   
\begin{align}
\label{PhilRep}
-\Phi(-G) f(x) = \int_0^\infty \big( Q_t f(x) -  f(x) \big) \phi(dt) 
\end{align}
given by Phillips where $\Phi$ has been introduced in \eqref{LevKinFormula}. This nice representation has been considered in \cite{Phillips52} and previously, for $G=\Delta$ in  \cite{Bochner49}. Indeed, in case of subordinate Brownian simigroups obtained via stable measure, we also refer to Bochner's subordination. For a general Markov process with generator $G$ we refer to Phillips. Let $\Psi$ be the Fourier multiplier of $G$. Then the Fourier symbol of the semigroup $Q_t = e^{tG}$ is written as $\widehat{Q_t} = e^{-t \Psi}$. For a well-defined function $f$, from \eqref{LevKinFormula} we have that
\begin{align*}
\int_{\mathbb{R}} e^{i\xi x} \left( -\Phi(-G) f(x) \right) dx 
= &  \left( \int_0^\infty \left( e^{-t \Psi(\xi)} -1 \right) \phi(dt) \right) \widehat{f}(\xi)\\
= & -\Phi \circ \Psi(\xi)\, \widehat{f}(\xi).
\end{align*}
Thus, $\Phi \circ \Psi$ is the Fourier multiplier of $A=-\Phi(-G)$.

\begin{remark}
The subordinator $H$ with symbol $\Phi$ has a density $h$ satisfying 
$$\partial_t h = - \Phi(\partial_x) h.$$
See the Phillips' representation with the translation semigroup $$e^{-t\partial_x}v(x)=v(x-t).$$
\end{remark}

\begin{remark}
We observe that
\begin{align*}
D^\alpha_x u(x) = \Phi(\partial_x) u(x) - \frac{x^{-\alpha}}{\Gamma(1-\alpha)} u(0).
\end{align*}
Thus, $\Phi(\partial_x) u(x)$ gives a Riemann-Liouville operator.
\end{remark}

\begin{remark}
(Dirichlet boundary condition) In a general setting, for compact domains, we have the compact representation
\begin{align}
\label{spectLap}
\Phi(-\Delta_\Omega) v(x) = \sum_{k} \Phi(\mu_k ) (v, e_k)\, e_k(x), \quad x \in \Omega
\end{align}
with $v(x)=0$ if $x \in \partial \Omega$. A sketch of proof can be given by considering that
\begin{align*}
Q_t e_k(x) = e^{-t \,\mu_k} e_k(x).
\end{align*} 
From the Phillips' representation and the fact that
\begin{align*}
v(x) = \sum_k (v, e_k)\, e_k(x)
\end{align*}
we get \eqref{spectLap}. Then, the equation is satisfied term by term, the convergence is easily verified by considering the suitable space $H^s$ of functions $v\in L^2$ such that $\sum_k (\mu_k)^{2s} (v,e_k) < \infty$. Observe that $\Phi(\mu_k)$ is written as in \eqref{LevKinFormula}, thus the spectral representation \eqref{spectLap} coincides with \eqref{PhilRep}.
\end{remark}

\begin{remark}
We note that if $H$ is the stable subordinator with symbol  $\Phi(\xi)=\xi^\alpha$ and  $Q_t$ is the semigroup of a Brownian motion on $\mathbb{R}^d$ with $\Psi(\xi)= |\xi|^2$,  than $|\xi |^{2\alpha}$ is the symbol $\Psi \circ \Phi$ of the fractional Laplacian
\begin{align}
\label{fracLaplacianS}
-\Phi(-G) = -(-\Delta)^{\alpha}, \quad \alpha\in (0,1)
\end{align}
where $(-\Delta)^\alpha u$ is defined as Cauchy principal value. The Phillips' representation \eqref{fracLaplacianS} holds only on $\mathbb{R}^d$, $d\geq 1$. We underline that the Phillips' representation on $\Omega \subset \mathbb{R}^d$ coincides with the spectral definition $-(-\Delta_\Omega)^\alpha$ where $(\Delta_\Omega, D(\Delta_\Omega))$ is the generator of a Brownian motion on $\Omega$. Thus, we obtain $-\Phi(-\Delta_\Omega)$  as in \eqref{spectLap} in case $\Delta_\Omega$ is the Dirichlet Laplacian.
\end{remark}

\subsection{Non-local (time) operators}
\label{sec:timeNLop}

Let $M>0$ and $w\geq 0$. Let $\mathcal{M}_w$ be the set of (piecewise) continuous functions on $[0, \infty)$ of exponential order $w$ such that $|u(t)| \leq M e^{wt}$. Denote by $\widetilde{u}$ the Laplace transform of $u$. Then, we define the operator $\mathfrak{D}^\Phi_t : \mathcal{M}_w \mapsto \mathcal{M}_w$ such that
\begin{align}
\label{lapDop}
\int_0^\infty e^{-\lambda t} \mathfrak{D}^\Phi_t u(t)\, dt = \Phi(\lambda) \widetilde{u}(\lambda) - \frac{\Phi(\lambda)}{\lambda} u(0), \quad \lambda > w
\end{align}
where $\Phi$ is given in \eqref{LevKinFormula}. Since $u$ is exponentially bounded, the integral $\widetilde{u}$ is absolutely convergent for $\lambda>w$.  By Lerch's theorem the inverse Laplace transforms $u$ and $\mathfrak{D}^\Phi_tu$ are uniquely defined. We note that
\begin{align}
\label{PhiConv}
\Phi(\lambda) \widetilde{u}(\lambda) - \frac{\Phi(\lambda)}{\lambda} u(0) = & \left( \lambda \widetilde{u}(\lambda) - u(0) \right) \frac{\Phi(\lambda)}{\lambda}
\end{align}
and thus, $\mathfrak{D}^\Phi_t$ can be written as a convolution involving the ordinary derivative and the inverse transform of \eqref{tailSymb} iff $u \in \mathcal{M}_w \cap C([0, \infty), \mathbb{R}_+)$ and $u^\prime \in \mathcal{M}_w$. By Young's inequality\footnote{(Young's convolution inequality.) Suppose that $f \in L_p(\mathbb{R})$, $g \in L_q(\mathbb{R})$ and $1/p + 1/q = 1/r+1$ with $1\leq p,q,r \leq \infty$. Then
\begin{align*}
\| f*g\|_r \leq \|f\|_p\, \|g\|_q.
\end{align*}
Now set $f=u^\prime$ and $g=\overline{\phi}$. Take the limit for $\lambda\to 0^+$ of the Laplace transform $\widetilde{g}(\lambda)$ in order to obtain $\|g\|_1$.} we also observe that
\begin{align}
\label{YoungSymb}
\int_0^\infty |\mathfrak{D}^\Phi_t u |^p dt \leq \left( \int_0^\infty |u^\prime |^p dt \right) \left( \lim_{\lambda \downarrow 0} \frac{\Phi(\lambda)}{\lambda} \right)^p, \qquad p \in [1, \infty)
\end{align}
where $\lim_{\lambda \downarrow 0} \Phi(\lambda) /\lambda$ is finite only in some cases. We notice that when $\Phi(\lambda)=\lambda$ (that is we deal with the ordinary derivative) we have that $H_t = t$ and $L_t=t$ a.s. and in \eqref{YoungSymb} the equality holds. We observe that $\mathfrak{D}^\Phi_t u$ is well-defined for $u$ such that $u,u^\prime \in \mathcal{M}_w$. For $w=0$, $u \in C_b([0, \infty))$ is a Lipschitz function on $[0, \infty)$. Indeed, we are considering $u$ with (piecewise) continuous and bounded derivative. Formula \eqref{YoungSymb} gives the relevant information stated in the following propositions.

\begin{prop}
\label{PropYoung1}
Let $\Phi$ be such that
\begin{align*}
\lim_{\lambda \downarrow 0} \frac{\Phi(\lambda)}{\lambda}  = \Phi_0 < \infty.
\end{align*}
If $u \in AC([0, \infty))$, then $\mathfrak{D}^\Phi_t u \in L_1([0, \infty))$.
\end{prop}
\begin{proof}
The statement follows from \eqref{YoungSymb} for $p=1$.
\end{proof}

\begin{prop}
\label{PropYoung2}
Let us consider the parabolic problem $\mathfrak{D}^\Phi_t u = A u$ for some generator $(A,D(A))$ and a given $\Phi$. Then, the existence of the corresponding elliptic problem $-A\bar{u} = \Phi_0 f$ is given according with  \eqref{YoungSymb}.
\end{prop}
\begin{proof}
Let us consider $\bar{u}(x) = \int_0^\infty u(t,x)dt$ where $u$ is the solution to $\mathfrak{D}^\Phi_t u = Au$ with $u_0=f \in D(A)$. Set $F(\lambda) := \int_0^\infty e^{-\lambda t} \mathfrak{D}^\Phi_t u\, dt$. Under the assumptions in Proposition \ref{PropYoung1}, as $\lambda \downarrow 0$, there exists $\Phi_0 = \Phi^\prime(0)$ such that $F \to -\Phi_0 f$ and $-A\bar{u}= \Phi_0\, f$. 
\end{proof}

\begin{remark}
\label{remark:telegr}
We recall that for $\Phi(\lambda)=\lambda^\alpha$, the symbol of a stable subordinator of order $\alpha$, the operator $\mathfrak{D}^\Phi_t$ becomes the Caputo-Dzherbashian derivative
\begin{align*}
\mathfrak{D}^\Phi_t u(t) = \frac{1}{\Gamma(1-\alpha)} \int_0^t \frac{u^\prime(s)}{(t-s)^\alpha}ds =: D^\alpha_t u(t)
\end{align*}
with $u^\prime(s)=du/ds$. Notice that for $u(s)=s^{\alpha} /\Gamma(\alpha + 1)$, we get that
\begin{align*}
\mathfrak{D}^\Phi_t u(t) = 1.
\end{align*}
\end{remark}

\begin{remark}
A further example is given by the symbol $\Phi(\lambda)=\lambda^{2\alpha} + \lambda^\alpha$ for $\alpha \in (0, 1/2)$. $\mathfrak{D}^\Phi_t$ becomes the fractional telegraph operator
\begin{align*}
\mathfrak{D}^\Phi_t u =  D^{2\alpha}_t u + D^\alpha_t u.
\end{align*}
The telegraph equation 
\begin{align*}
\frac{\partial^2 u}{\partial t^2} + \frac{\partial u}{\partial t} = \Delta u 
\end{align*}
is associated with the telegraph process 
\begin{align*}
X_t = \int_0^t V_s \, ds = \int_0^t (-1)^{N_s} ds
\end{align*}
where $N_s$ is a Poisson process. We are able to write the corresponding Langevin equation for the position $X_t$ and the velocity $V_t$. The fractional telegraph equation and the probabilistic representation of the solution has been investigated in \cite{DovTO14a, DovTO14b} and in a general setting in \cite{DovPol17}. The probabilistic representation of the fractional telegraph equation has an interesting connection with the inverse of the sum of independent stable subordinators. A further reading in case of higher dimension can be found in \cite{hTelegraph}.
\end{remark}

The operator $\mathfrak{D}^\Phi_t$ (and the associated non-local Cauchy problem) has been considered, with some alternative representation in \cite{Koc2011, Toaldo2015, Chen2017, Ascione} reported here in chronological order. Concerning the fractional Cauchy problem we also recall the works \cite{BM2001, DovSPA, DovNan, MNV09}. Further references will be given below for the non-local Cauchy problem.

\subsection{Non-local operators and subordinators}
\label{Sec:GenNLops}
Let us write the equations governing $H$ and $L$ in case $\Phi$ does not include the time-dependent-continuity for $h$ (see \cite{ColDov1, ColDov2}). In particular, we consider densities associated with $\Phi$ such that
\begin{align*}
h(t,0) = \lim_{\lambda \to \infty} \lambda e^{-t\Phi(\lambda)} = \lim_{\lambda\to \infty} \exp \left(\ln \lambda - t  \Phi(\lambda) \right) = 0 \quad \forall\, t>0.
\end{align*}
That is,
\begin{align*}
\forall\, t>0, \quad t \Phi(\lambda) - \ln \lambda \to \infty  \quad \textrm{as} \quad \lambda \to \infty.
\end{align*}
For our convenience here we write $u_f$ for a function $u=u(t,x)$ such that $u(t,x)\to f(x)$ as $t\downarrow 0$. We also recall that, for $p\geq 1$, $W^{1,p}$ is the set of functions in $L^p$ with first derivative in $L^p$. The set $W^{1,p}_0$ is the collection of functions in $W^{1,p}$ extended with zero on $(-\infty, 0]$. The function
\begin{align*}
h_f \in C^{1,1}((0,\infty), W^{1,1}_0(0,\infty))
\end{align*}
written as
\begin{align*}
h_f(t,x) 
= & \int_0^x f(x-y) h(t,y)dy = \mathbf{E}_0[f(x-H_t) \mathbf{1}_{(t< L_t)}]
\end{align*}
is the unique solution to
\begin{equation}
\begin{cases}
\displaystyle \frac{\partial}{\partial t} h_f(t,x) = - \mathbf{D}^\Phi_{x-} h_f(t,x), \quad (t,x) \in (0, \infty) \times (0, \infty)\\
\displaystyle h_f(t,0)=0, \quad t>0\\
\displaystyle h_f(0,x)=f(x) \in W^{1,1}_0(0,\infty)
\end{cases}
\label{eqHf}
\end{equation}
and, the function
\begin{align*}
l_f \in C^{1,1}(W^{1,\infty}(0,\infty), (0,\infty))
\end{align*}
written as
\begin{align*}
l_f(t,x) 
= & \int_0^x f(x-y) l(t,y)dy = \mathbf{E}_0[f(x-L_t) \mathbf{1}_{(t< H_t)}]
\end{align*}
is the unique solution to
\begin{equation}
\begin{cases}
\displaystyle \mathfrak{D}^\Phi_t l_f(t,x) = - \frac{\partial}{\partial x}l_f(t,x), \quad (t,x) \in (0, \infty) \times (0, \infty)\\
\displaystyle l_f(t,0)=0, \quad t>0\\
\displaystyle l_f(0,x)=f(x) \in L^1(0,\infty).
\end{cases}
\label{eqLf}
\end{equation}
Notice that in the latter problem, as $f\in L^p(0,\infty)$ with $p\geq 1$, we get $l_f(t,\cdot) \in L^p(0,\infty)$ $\forall\, t>0$ (just by applying the Young's inequality). Moreover, we observe that $h_f$ and $l_f$ may be regular more that $C^{1,1}$.

Let us write
\begin{equation}
\bar{h}_f(x)= \int_0^\infty h_f(t,x)dt \quad \textrm{and} \quad \bar{l}_f(x)= \int_0^\infty l_f(t,x)dt.
\end{equation}
We can immediately check that the Abel (type) equation $f(x) = \mathbf{D}^\Phi_{x-} \bar{h}_f(x)$ gives the elliptic problem associated with \eqref{eqHf}. On the other hand, the elliptic problem associated with \eqref{eqLf} exists only if $\lim_{\lambda \downarrow 0} \Phi(\lambda)/\lambda < \infty$. Such a fact is not surprising, indeed by considering $f=\mathbf{1}$,
\begin{align*}
\bar{l}_\mathbf{1}(x)= \int_0^\infty \mathbf{E}_0 [\mathbf{1}_{(t< H^\Phi_x)}] dt = \mathbf{E}_0[H^\Phi_x]= x \lim_{\lambda \downarrow 0} \frac{\Phi(\lambda)}{\lambda}.
\end{align*}
In case the elliptic problem exists, it takes the form
\begin{align*}
f(x) =   \left( \lim_{\lambda \downarrow 0} \frac{\lambda}{\Phi(\lambda)}\right) \frac{\partial}{\partial x} \bar{l}_f(x).
\end{align*}

Now we move to the PDEs connection concerned respectively with non-local equations and non-local boundary conditions. 

\section{Non-local initial value problems}

First we recall that for $u(t,x)=Q_t f(x)$ solving the problem
\begin{align*}
\frac{\partial u}{\partial t} = G u, \quad u_0=f \in D(G)
\end{align*}
we have the probabilistic representation
\begin{align*}
u(t,x)= \mathbf{E}_x[f(X_t)]
\end{align*}
where $X=\{X_t\}_{t\geq 0}$ on $E$ has generator $(G, D(G))$. Then, we define the time-changed process 
$$X^{L}_t :=X_{L_t}, \quad t \geq 0$$  
on $E$ as the composition $X \circ L$ and the time-changed process 
$$X^H_t := X_{H_t}, \quad t>0$$ 
on $E$ as the composition $X \circ H$. The governing equations of $X^L_t$ and $H^H_t$ together with their properties have been extensively investigated in last years.


\subsection{Parabolic problems}
Let us consider the time non-local Cauchy problem
\begin{align}
\label{time-frac-problem}
\mathfrak{D}^\Phi_t u = G u, \qquad u_0=f \in D(G).
\end{align}
The probabilistic representation of the solution to \eqref{time-frac-problem} is written in terms of the time-changed process $X^{L} = \{X^L_t\}_{t\geq 0}$, that is 
\begin{align}
\label{sol-time-frac-problem}
u(t,x) = \mathbf{E}_x[f(X^{L}_t)]
\end{align}
and
\begin{align*}
\mathbf{E}_x[f(X^{L}_t)]= \mathbf{E}_x[f({^*X}^{L}_t), t < \zeta^{L}] :=Q^{L}_t f(x) 
\end{align*}
where $\zeta^{L}$ is the lifetime of $X^{L}$, the part process of ${^*X}^{L} = \{{^*X}^{L}_t\}_{t\geq0}$ on $E$.  The fact that $L$ is continuous implies that 
\begin{align}
\label{semigPhi-L-continuity}
\mathbf{E}_x[f({^*X}^{L}_t), t < \zeta^{L}] = \mathbf{E}_x[f({^*X}_{L_t}), \, L_t < \zeta] =  \int_0^\infty Q_s f(x) \mathbf{P}_0(L_t \in ds) 
\end{align}
where $Q_tf(x) = \mathbf{E}_x [f(X_t)] = \mathbf{E}_x [f({^*X}_t), t < \zeta]$ and $\zeta$ is the lifetime of $X$, the part process of ${^*X}=\{ {^*X}_t\}_{t\geq 0}$ on $E$. Indeed, $X$ on $E$ (and therefore $X^L$ on $E$) is obtained by killing ${^*X}$ on $E^* \supset E$ (${^*X}^{L}$ on $E^* \supset E$). For the  time-changed process we have that
\begin{align}
\label{semigPhi-1}
Q^{L}_t \mathbf{1}_E = \mathbf{P}_x(t < \zeta^{L}) = \mathbf{P}_x(L_t < \zeta) = \int_0^\infty \mathbf{P}_x(s < \zeta) \mathbf{P}_0(L_t \in ds).
\end{align}
We notice that $Q_t^{L}$ is not a semigroup (indeed the random time is not Markovian). For the process $X$ with generator $(G, D(G))$  and the independent subordinator $H$ with symbol \eqref{LevKinFormula}, the process $X^H_t=X_{H_t}$, $t\geq 0$ can be considered in order to solve the problem
\begin{align}
\label{space-frac-problem} 
\frac{\partial u}{\partial t} = - \Phi( - G) u, \quad u_0=f \in D(\Phi( - G)) \subset D(G).
\end{align} 
The probabilistic representation of the solutions to \eqref{space-frac-problem} is given by
\begin{align*}
u(t,x) =\mathbf{E}_x [f(X^{H}_t)]
\end{align*}
and
\begin{align*}
\mathbf{E}_x [f(X^{H}_t)] = \mathbf{E}_x [f({^*X}^{H}_t), t < \zeta^{H}] =: Q^{H}_t f(x)
\end{align*}
where $\zeta^{H}$ is the lifetime of $X^{H}$, that is the part process of ${^*X}^{H}$ on $E$. We notice that if $H$ is a stable subordinator with symbol $\Phi(\lambda)=\lambda^\alpha$ and $X$ is a Brownian motion, then $-\Phi(-G)$ is the fractional Laplacian. Since $H$ is not continuous, $X^H$ may have jumps. We underline the fact that $\Phi(-G)$ is a generator of a Markov process. Indeed,  $X^H$ is the composition of Markov processes. Thus, we may consider also $A=-\Phi(-G)$ in the next theorem.

Now we state the following result (see \cite[Theorem 5.2]{CapDovFCAA}). 

\begin{theorem}
\label{time-frac-THM}
Let $(A, D(A))$ be the generator of the Feller process $X=(X_t, t\geq 0)$ on $E$. Then, $u(t,x) = \mathbf{E}_x[f(X^L_t)]$, $t\geq0$, $x \in E$ is the unique strong solution in $L^2(E)$ to 
\begin{align}
\label{time-frac-problem-A}
\mathfrak{D}^\Phi_t u = Au, \quad u_0=f \in D(A)
\end{align}
in the sense that:
\begin{enumerate}
\item $\varphi: t \mapsto u(t, \cdot)$ is such that $\varphi \in C([0, \infty), \mathbb{R}_+)$ and $\varphi^\prime  \in \mathcal{M}_0$,
\item $\vartheta : x \mapsto u(\cdot, x)$ is such that $\vartheta, A\vartheta \in D(A)$,
\item $\forall\, t > 0$, $\mathfrak{D}^\Phi_t u(t,x) = A u(t,x)$ holds a.e in $E$,
\item $\forall\, x \in E$, $u(t,x) \to f(x)$ as $t \downarrow 0$.
\end{enumerate}
\end{theorem}

We underline that there is a consistent literature on this topic, therefore the references are intended to be illustrative, and not exhaustive. In \cite{MShef08} the authors study mild solutions (in a sense specified in the paper) for space-time pseudo-differential equations. In \cite{Toaldo2015} the time-changed $C_0$-semigroup has been investigated. Similarly, in \cite{Chen2017} the author proves existence and uniqueness of strong solutions to general time fractional equations with initial data $f \in D(A)$. In \cite{CKKW} the authors establish existence and uniqueness for weak solutions and initial data $f \in L^2$.

\subsection{Elliptic problems}
We proceed with a short discussion of results which will be useful below. The solution to
\begin{align*}
-Gu = f \quad on\ E
\end{align*}
is given by
\begin{align*}
u(x) = \mathbf{E}_x\left[\int_0^{\zeta} f(X_s)ds \right]
\end{align*}
where $\zeta$ is the lifetime of $X$. We may write
\begin{align*}
u(x) = \int_0^\infty Q_t f(x) dt
\end{align*}
which plays an interesting role in the discussion of elliptic problems for time-changed processes (Proposition \ref{PropYoung1} and Proposition \ref{PropYoung2}). The reader should have in mind formula \eqref{YoungSymb}. If $f=\mathbf{1}_E$, then
\begin{align*}
Q_t \mathbf{1}_E(x) = \mathbf{P}_x(X_t \in E) = \mathbf{P}_x(t < \tau_E)
\end{align*}
where $\tau_E :=\inf\{s>0\,:\, X_s \notin E\}$ is the first exit time of $X$ from $E$ and
\begin{align*}
\int_0^\infty \mathbf{P}_x(t < \tau_E) dt = \mathbf{E}_x[\tau_E]
\end{align*}
is the mean exit time of $X$ from $E$, that is the solution to
\begin{align*}
-G u = \mathbf{1}_E.
\end{align*}
We move to the non-local problems.\\

Concerning the process $X^H_t$, the solution to
\begin{align*}
\Phi(-G) u = \mathbf{1}_E
\end{align*}
is given by
\begin{align*}
u(x) = \mathbf{E}_x[\tau_E^H]
\end{align*}
where $\tau_E^H :=\inf\{s>0\,:\, X^H_s \notin E\}$ is the first exit time of $X^H$ from $E$. \\

Concerning the process $X^L_t$, the solution to
\begin{align*}
-Gu= \Phi^\prime(0) \mathbf{1}_E
\end{align*}
is given by
\begin{align*}
u(x) = \mathbf{E}_x[\tau_E^L]
\end{align*}
where $\tau_E^L :=\inf\{s>0\,:\, X^L_s \notin E\}$ is the first exit time of $X^L$ from $E$. Moreover,
\begin{align*}
\mathbf{E}_x[\tau_E^L] = \Phi^\prime(0)\, \mathbf{E}_x[\zeta]
\end{align*}
where $\zeta$ is the lifetime of the base process $X_t$. If we have Dirichlet boundary condition for example, then $\zeta = \tau_E$. We recall that 
\begin{align*}
\Phi^\prime(0) = \lim_{\lambda \to 0} \frac{\Phi(\lambda)}{\lambda}
\end{align*}
is finite only in some cases and determines the delay of the base process in a sense to be better specified below. We discuss in detail this relation in the next section. Here we only anticipate the following cases as introductory examples. If we have $\Phi(\lambda)=\lambda^\alpha$, then the inverse $L_t$ to a stable subordinator slows down $X_t$ and the composition $X^L_t$ turns out to be delayed. Indeed, we obtain that $\mathbf{E}_x[\zeta^L]=\infty$. On the other hand, for $a>0$, $b>0$ 
\begin{align*}
\Phi(\lambda) = a\ln (1+\lambda/b) = \int_0^\infty (1- e^{-\lambda y}) a e^{-b y} \frac{dy}{y}
\end{align*}
is the symbol of the gamma subordinator $H$ with density 
$$\mathbf{P}_0(H_t \in dx) =  \frac{b^{at}}{\Gamma(at)} x^{at - 1} e^{-b x} dx$$ 
and 
\begin{align*}
\mathbf{E}_0[e^{-\lambda H_t}] = \int_0^\infty e^{-\lambda x} \mathbf{P}_0(H_t \in dx) = \left( 1 + \lambda/b \right)^{-at} = e^{- t a \ln (1+ \lambda /b)}.
\end{align*}
We have that
\begin{align*}
\Phi^\prime(0) = \frac{b}{a} < \infty
\end{align*}
and the mean value of the lifetime $\zeta^L$ is finite. The time-changed process $X^H_t$ is known as variance gamma process or Laplace motion (see \cite{ColDov1} and the references therein). It is a L\'evy process with no diffusion component (a pure jump process), the increments are independent and follow a variance-gamma distribution, which is a generalization of the Laplace distribution (modified Bessel function $K$).

\subsection{Delayed and rushed processes}

We start with the following definition given in \cite{BCM} for the Brownian motion.
\begin{definition}
\label{def1}
Let $E \subset \mathbb{R}^d$ be an open connected set with finite volume. Let $B \subset E$ be a closed ball with non-zero radius. Let $X$ be a reflected Brownian motion on $\overline{E}$ and denote by $T_B= \inf\{t\geq 0\,:\, X_t \in B\}$ the first hitting time of $B$ by $X$. We say that $E$ is a \emph{trap domain} for $X$ if 
\begin{align*}
\sup_{x \in E} \mathbf{E}_x[T_B] = \infty.
\end{align*}
Otherwise, we say that $E$ is a \emph{non-trap domain} for $X$.
\end{definition}
In Definition \ref{def1}, the random time $T_B$ plays the role of lifetime for the Brownian motion on $\bar{E} \setminus \bar{B}$ reflected on $\partial E \setminus \partial B$ and killed on $\partial B$.

Further on we denote by $\zeta$ (possibly with some superscript) the lifetime of a process $X$, that is for the process $X_t$ in $E$ with $X_0=x \in E$ (denote by $E^c$ the complement set of $E$),
\begin{align*}
\zeta := \inf\{t >0\,:\, X_t \in E^c\}.
\end{align*}
Let $T$ be a random time and denote by $X^T := X \circ T$ the process $X$ time-changed by $T$. It is well-known that $X^T$ is Markovian only for a Markovian time change $T$, otherwise from the Markov process $X$ we obtain a non-Markov process $X^T$. Denote by $\zeta^T$ the lifetime of $X^T$. We consider the following characterization in terms of lifetimes given in \cite{CapDovDelRush}. 
\begin{definition}
\label{defDelRus}
Let $E \subset \mathbb{R}^d$. 
\begin{itemize}
\item[-] We say that $X$ is delayed by $T$ if $\mathbf{E}_x[\zeta^T] > \mathbf{E}_x[\zeta]$, $\forall \, x \in E$.
\item[-] We say that $X$ is rushed by $T$ if $\mathbf{E}_x[\zeta^T] < \mathbf{E}_x[\zeta]$, $\forall \, x \in E$.
\end{itemize}
Otherwise, we say that $X$ runs with its velocity.
\end{definition}

\begin{remark}
Let $X$ be a Brownian motion. If $X$ is killed on $\partial E$ we notice that $\mathbf{E}_x[\zeta] < \infty$. We underline the fact that if $X$ is reflected on $\partial E \setminus \partial B$ and killed on $\partial B$ with $B \subset E$, we have that $\sup_x \mathbf{E}_x[\zeta] < \infty$ only if $E$ is non-trap for $X$. Examples of non-trap domains are given by smooth domains and snowflakes domains as the  scale irregular (Koch) fractals in figures \ref{fig1} and \ref{fig2} (see \cite{CAP} for details). Figures \ref{fig1} and \ref{fig2} are realizations of the random domain obtained by choosing randomly the contraction factor step by step in the construction of the pre-fractals.
\end{remark}

\begin{figure}
\includegraphics[scale=1]{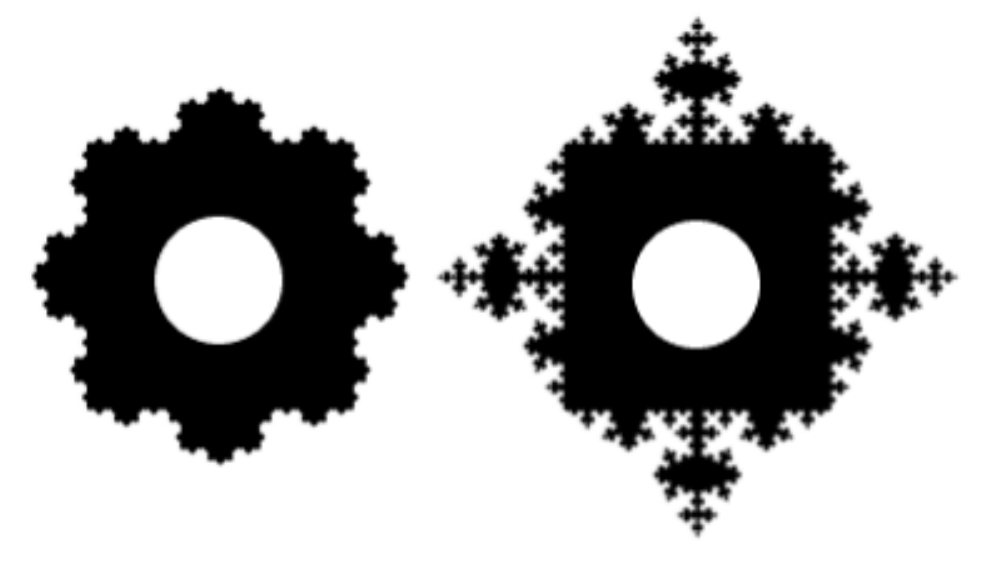} 
\caption{Koch curves outside the square. We have Neumann condition 
on $\partial E \setminus \partial B$ and Dirichlet condition on $\partial B$ where $B$ is the Ball inside $E$. The domains are non-trap for the Brownian motion. 
}
\label{fig1}
\end{figure}

\begin{figure}
\includegraphics[scale=1]{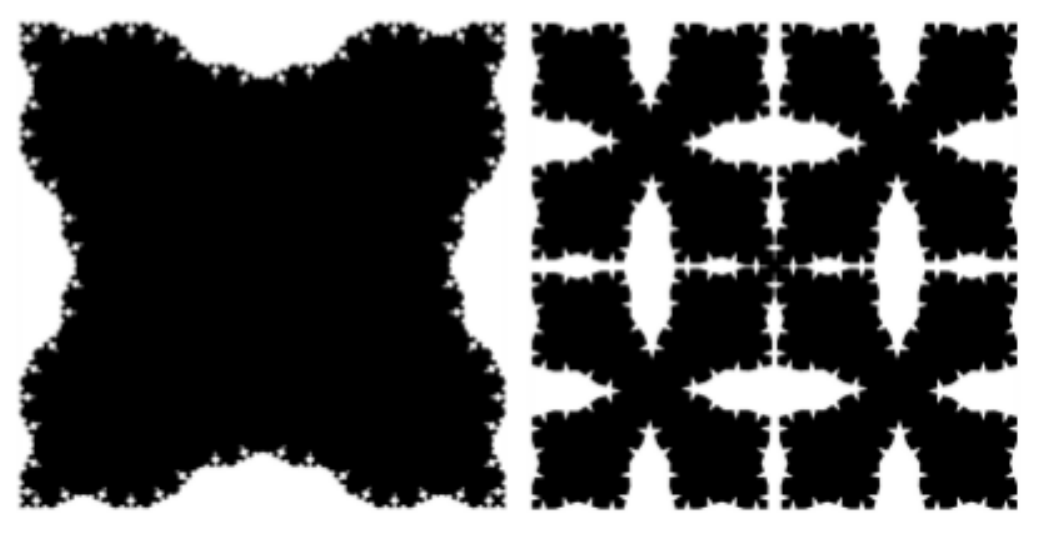} 
\caption{Koch curves inside the square. We have Dirichlet condition on the boundary $\partial E$. The domains are non-trap for the Brownian motion. 
}
\label{fig2}
\end{figure}

\begin{theorem}
\label{meanLzeta}
(\cite{CapDovDelRush}) Let $\Phi$ be the Bernstein function in \eqref{LevKinFormula}. Let $H$ be the subordinator with symbol $\Phi$. We have
\begin{align}
\mathbf{E}_x[\zeta^L] = \mathbf{E}_x[H_\zeta], \quad x \in E.
\end{align}
\end{theorem}

\begin{theorem}
\label{meanHzeta}
(\cite{CapDovDelRush}) Let $\Phi$ be the Bernstein function \eqref{LevKinFormula}. Let $L$ be the inverse to a subordinator $H$ with symbol $\Phi$.  We have that 
\begin{align}
\mathbf{E}_x[\zeta^H] = \mathbf{E}_x[L_\zeta], \quad x \in E.
\end{align}
\end{theorem}

The previous results must be read by considering that
\begin{equation}
\mathbf{E}_x[H_\zeta] = \lim_{\lambda \downarrow 0} \frac{\Phi(\lambda)}{\lambda} \mathbf{E}_x[\zeta]
\end{equation}
whereas $\mathbf{E}_0[L_t]$ is known only on some cases (see for example \cite{ColDov1}).

The time change may lead to unexpected situation. The gamma subordinator represents a nice example in this direction. Let $H$ be characterized by $\Phi(\lambda) = a \ln (1+\lambda/b)$ with $0< a,b < \infty$. Then, $\Phi(\lambda)/\lambda \to b/a$ for $\lambda \downarrow 0$ as mentioned above. This means that, for $b<a$ for example, the base process $X$ is rushed by $L$ and the process $X^L$ can not be considered as a delayed process. We have a delaying effect for $b>a$. The process $X^L$ is actually a delayed process in case $L$ is an inverse to a stable subordinator (with $\Phi(\lambda)=\lambda^\alpha$). Indeed, the lifetime $\zeta^L$ turns out to be finite with some probability but its mean value is infinite. The Definition \ref{defDelRus} must be therefore understood by considering all possible paths of a time-changed process. This means that we have \lq\lq delayed in mean\rq\rq\ or \lq\lq rushed in mean\rq\rq processes.

\section{An irregular domain}

Our aim is to underline some relation between non-local operators and the \lq\lq regularity\rq\rq\ of the domains. Such a connection is given in terms of stochastic processes describing the motions in that domains. In order to be clear we present here a discussion based on a simple and instructive case, the fractal Koch domain.

\subsection{Random Koch domains (RKD)}
\label{section:RKD}
Let $\ell_{a} \in (2,4)$ with $a \in I \subset \mathbb{N}$ be the reciprocal of the contraction factor for the family $\Psi^{(a)}$ of contractive similitudes $\psi^{(a)}_i : \mathbb{C} \to \mathbb{C}$ given by
\begin{align*}
\psi^{(a)}_1(z)=\frac{z}{\ell_a}, \quad \psi^{(a)}_2(z)= \frac{z}{\ell_a} e^{\imath \theta(\ell_a)} + \frac{1}{\ell_a},
\end{align*}
\begin{align*}
\psi^{(a)}_3(z) = \frac{z}{\ell_a} e^{\imath \theta(\ell_a)} + \frac{1}{2} + \imath \sqrt{\frac{1}{\ell_a} - \frac{1}{4}}, \quad \psi^{(a)}_4(z) = \frac{z-1}{\ell_a} + 1
\end{align*}
where $\theta(\ell_a) = \arcsin(\sqrt{\ell_a(4-\ell_a)}/2)$. Let $\Xi =I^\mathbb{N}$ with $I \subset \mathbb{N}$, $|I|=N$, and let $\xi=(\xi_1, \xi_2, \ldots) \in \Xi$.  We call $\xi$ an environment sequence where $\xi_n$ says which family of contractive similitudes we are using at level $n$. Set $\ell^{(\xi)}(0)=1$ and 
\begin{align*}
\ell^{(\xi)}(n)= \prod_{i=1}^n \ell_{\xi_i}.
\end{align*}
We define a left shift $S$ on $\Xi$ such that if $\xi = \left( \xi_1, \xi_2,\xi_3, \ldots \right),$ then $S\xi = \left(  \xi_2,\xi_3, \ldots \right).$ For $B\subset \mathbb{R}^2$ set $\Upsilon^{(a)}(B) = \bigcup_{i =1}^{4} \psi^{( a)}_i \left( B \right)$ and $\Upsilon^{(\xi)}_n(B) =\Upsilon^{(\xi_1)}\circ\dots\circ\Upsilon^{(\xi_n)} \left( B \right)$. The fractal $K^{( \xi )}$ associated with the environment sequence $\xi$ is defined by
$$K^{( \xi )}=\overline{\bigcup_{n = 1}^{+ \infty}\Upsilon_n^{(\xi)}(\Gamma)}$$
where $\Gamma=\{ P_1, P_2\}$ with $P_1=(0,0)$ and $P_2=(1,0).$  For $\xi\in \Xi,$ we define the word space $W=W^{(\xi)}=\{ (w_1,w_2,...): 1\leqslant w_i\leqslant 4\}$ and, for $w\in W,$ we set $w|n=(w_1,...,w_n)$ and $\psi_{w | n}=\psi^{(\xi_1)}_{w_1}\circ\dots\circ\psi^{(\xi_n)}_{w_n}.$ The volume measure  $\mu^{(\xi)} $ is  the unique Radon measure on $K^{(\xi)}$ such that 
\begin{equation*}
\mu^{(\xi)}(\psi_{w | n}(K^{(S^n\xi)}))=\frac{1}{4^n}
\end{equation*} 
for all $w\in W,$ as,  for each $a\in A,$ the family $\Psi^{(a)} $ has $4$  contractive similitudes. Let $K_0$ be the line segment of unit length with $P_1=(0,0)$ and $P_2=(1,0)$  as endpoints. We set, for each $n \in \mathbb{N}$, 
\begin{equation*}
K^{(\xi)}_n=\Upsilon_n^{(\xi)}(K_0)
\end{equation*}
and ${K^{(\xi)}_n}$ is the so-called $n$-th \text{prefractal (deterministic) curve}.

Let us consider the random vector $\boldsymbol{\xi}= (\boldsymbol{\xi}_1, \boldsymbol{\xi}_2, \ldots)$ whose components $\boldsymbol{\xi}_i$ take values on $I$ with probability mass function ${\bf P} : \Xi \to [0,1]$. Thus, the construction of the random $n$-th pre-fractal  curve $K^{(\boldsymbol{\xi})}_n=\Upsilon_n^{(\boldsymbol{\xi})}(K_0) $ depends on the realization of $\boldsymbol{\xi}$ with probability $P(\boldsymbol{\xi}_i=\xi_i)$ for its $i$-th  component. We assume that $\{\boldsymbol{\xi}_i\}_{i=1, \ldots, n}$ are identically distributed and $\boldsymbol{\xi}_i \perp \boldsymbol{\xi}_j$ for $i\neq j$, that is we obtain the curve $K_n^{(\xi)}$ with probability
\begin{align*}
{\bf P}(\boldsymbol{\xi}| n =\xi | n) = \prod_{i=1}^n {\bf P}(\boldsymbol{\xi}_i=\xi_i)
\end{align*}
where $\boldsymbol{\xi}| n = (\boldsymbol{\xi}_1, \ldots, \boldsymbol{\xi}_n)$ and $\xi |n = (\xi_1, \ldots, \xi_n)$. Further on we only use the superscript $(\xi |n)$ or $(\boldsymbol{\xi}|n)$ in order to streamline the notation. \\

Given the random environment sequence $\boldsymbol{\xi}$, the random fractal $K^{(\boldsymbol{\xi})}$  is therefore defined by the deterministic fractal $K^{(\xi)}$.\\

Let $\Omega^{(\xi |n)}$ be the planar domain obtained from a regular polygon by replacing each side with a pre-fractal curve $K_n^{(\xi)}$ and $\Omega^{(\xi)}$ be the planar domain obtained by replacing each side with the corresponding fractal curve $K^{(\xi)}$. We introduce the random planar domains $\Omega^{(\boldsymbol{\xi}|n)}$ and $\Omega^{(\boldsymbol{\xi})}$ by considering the random curves $K_n^{(\boldsymbol{\xi})}$ and $K^{(\boldsymbol{\xi})}$. Examples of (pre-fractal) random Koch domains are given in Figure \ref{fig-outside} (outward curves), Figure \ref{fig-con} (inward curves) and Figure \ref{fig-con-2} (inward curves) by choosing the square as regular polygon. 

\begin{figure}
\centering
\includegraphics[scale=.6]{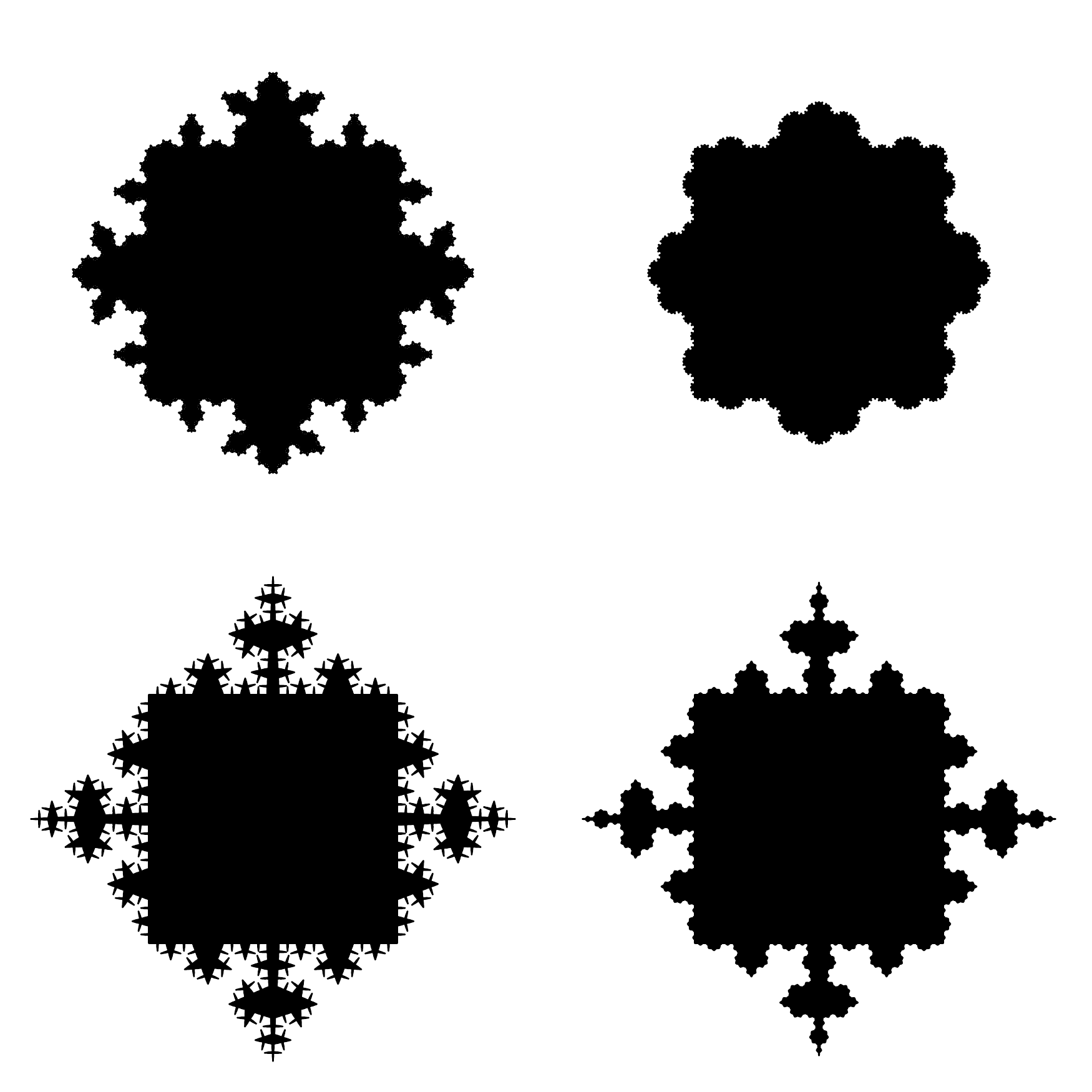}
\caption{Outward curves}
\label{fig-outside}
\end{figure}

\begin{figure}
\centering
\includegraphics[scale=1]{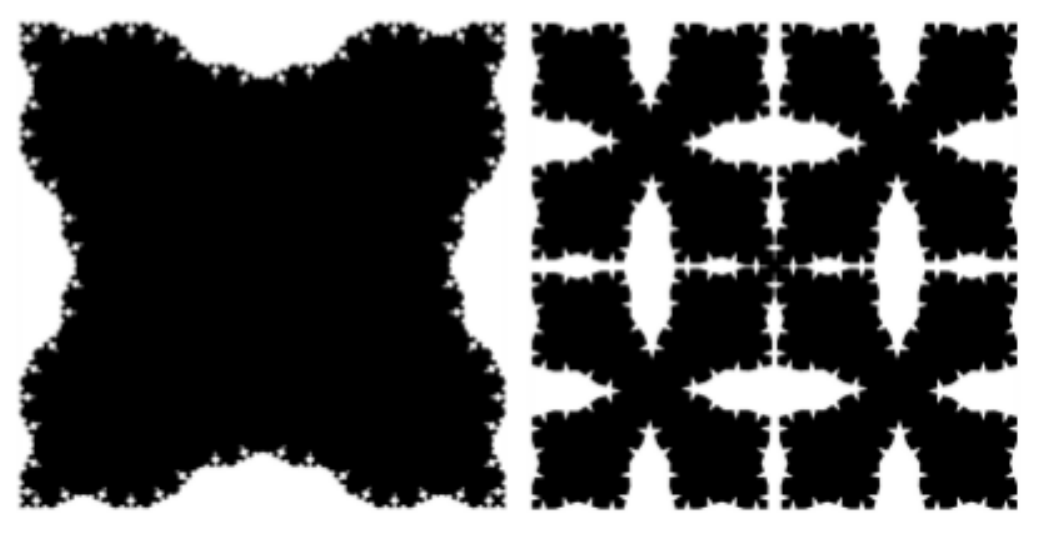}
\caption{Inward curves}
\label{fig-con}
\end{figure}

\begin{figure}
\centering
\includegraphics[scale=1]{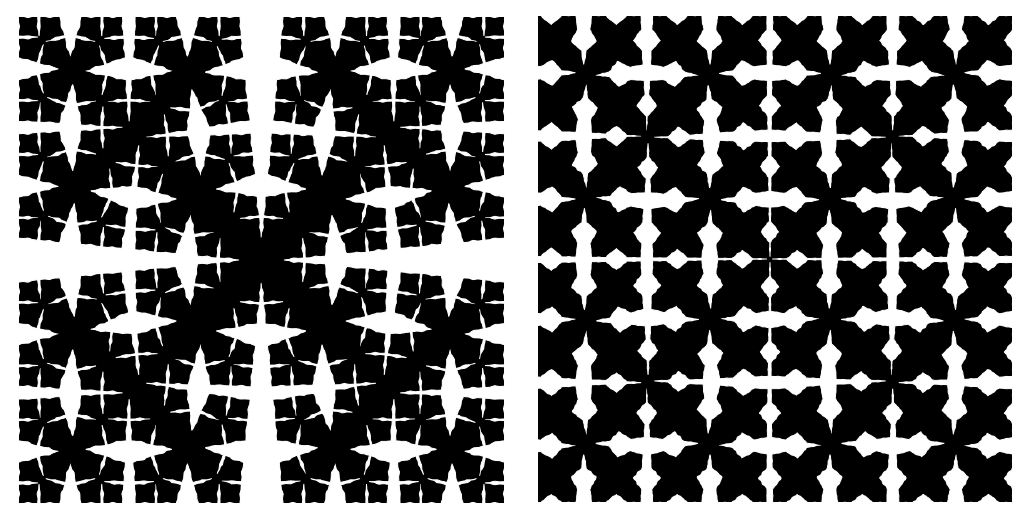}
\caption{Inward curves}
\label{fig-con-2}
\end{figure}

%

\begin{theorem}
For the sequence of (random) Hausdorff dimension $d^{(\boldsymbol{\xi}|n)}$ we have that
\begin{align*}
d^{(\boldsymbol{\xi}|n)} \stackrel{a.s.}{\to} d^{(\boldsymbol{\xi})} = \frac{\ln 4}{\mathbf{E}[\ln \ell_{\boldsymbol{\xi}_1}]}, \quad n \to \infty.
\end{align*}
\end{theorem}
\begin{proof}
Since $\boldsymbol{\xi}_i \stackrel{law}{=} \boldsymbol{\xi}_1$, $\forall\, i$, we have that the Hausdorff dimension $d^{(\boldsymbol{\xi})}$ of the curve  $K^{(\boldsymbol{\xi})}$ can be obtained by considering the strong law of large numbers and the fact that
\begin{align*}
\frac{\ln 4^n}{\sum_{i=1}^n \ell_{\boldsymbol{\xi}_i}}  = \frac{\ln 4}{\frac{1}{n} \sum_{i=1}^n \ell_{\boldsymbol{\xi}_i}} \stackrel{a.s.}{\to} \frac{\ln 4}{\mathbf{E}[\ln \ell_{\boldsymbol{\xi}_1}]}, \quad n \to \infty.
\end{align*}
Then (see \cite[Lemma 2.3]{BH}),
\begin{align}
d^{(\boldsymbol{\xi})} = \frac{\ln 4}{\ln \prod_{a \in I} (\ell_{a})^{P(\boldsymbol{\xi}_1 =a)}} = \frac{\ln 4}{\mathbf{E}[ \ln \ell_{\boldsymbol{\xi}_1}]}. \label{dimHK}
\end{align}
\end{proof}
The measure $\mu^{(\xi)}$ satisfies the following property. There exist two positive constants $C_1, C_2,$ such that,
\begin{equation}
\label{eq:6} 
C_1 r^{d^{(\xi)} }\le \mu^{{(\xi)} }(\mathcal{B}(P,r)\cap K^{(\xi)})\le C_2 r^{d^{(\xi)} }\ ,\quad
\forall\,P\in K^{(\xi)}, 
\end{equation} 
where $\mathcal{B}(P,r)$ denotes the Euclidean ball with center in $P$ and radius $0<r\leq1$ (see  \cite{BH}).  According to Jonsson and Wallin (\cite{JonWal}), we say that $K^{(\xi)} $ is a $d$-set with respect to  the Hausdorff measure $\mathcal{H}^d,$ with $d=d^{(\xi)}.$ The sequence
\begin{align}
\sigma^{(\xi | n)} = \frac{\ell^{(\xi | n)}}{4^n}, \qquad \textrm{where} \qquad \ell^{(\xi | n)} = \prod_{i=1}^n \ell_{\xi_i} \label{seq-sig}
\end{align}
is obtained from the realization of $\boldsymbol{\xi}|n$ and therefore, from the realization of the random variable $\ell^{(\boldsymbol{\xi}|n)}$ with mean value given by
\begin{align*}
\mathbf{E}[\ell^{(\boldsymbol{\xi} | n)}] = \prod_{i=1}^n \mathbf{E} [\ell_{\boldsymbol{\xi}_i}] = \left( \mathbf{E}[ \ell_{\boldsymbol{\xi}_1}] \right)^n.
\end{align*}
For $\alpha=\mathbf{E} \ell_{\boldsymbol{\xi}_1} \in (2,4)$ we find the mean value $\mathbf{E}[\sigma^{(\boldsymbol{\xi} | n)}] = \alpha^n / 4^n$ from \eqref{seq-sig}. In terms of that sequence we are able to deal with the pre-fractal (and fractal) boundary by introducing for example the Revuz measure associated with some boundary functional. In particular, we take into account the integral
\begin{align*}
\sigma^{(\xi |n)} \int_{\partial \Omega^{(\xi |n)}} f(x) \mathbf{1}_{\Lambda}(x) d\mathfrak{s}, \quad \Lambda \in \partial \Omega^{(\xi |n)}, \quad n \in \mathbb{N}
\end{align*} 
where $\mathfrak{s}$ is the arc-length measure on the pre-fractal boundary leading to the Hausdorff measure $\mu^{(\xi)}$ on the fractal boundary. As $n\to \infty$, $\sigma^{(\xi |n)}$ ensures convergence.

The random contraction factors give rise to random (pre-fractal) Koch domains $\Omega^{(\boldsymbol{\xi}|n)}$. The random (fractal) Koch domains $\Omega^{(\boldsymbol{\xi})}$ are obtained as described above. In order to have a clear picture about the relevance of the Koch domain in our analysis, we recall an example introduced in \cite{BCM}. Assume that the construction of a Koch curve $K^{(\xi)}$ is given by the picture in Figure \ref{Fig:KochModified}. The Koch domain becomes a trap domain for the Brownian motion according with the asymptotic behavior of the opening size.

\begin{figure}
\includegraphics[scale=.5]{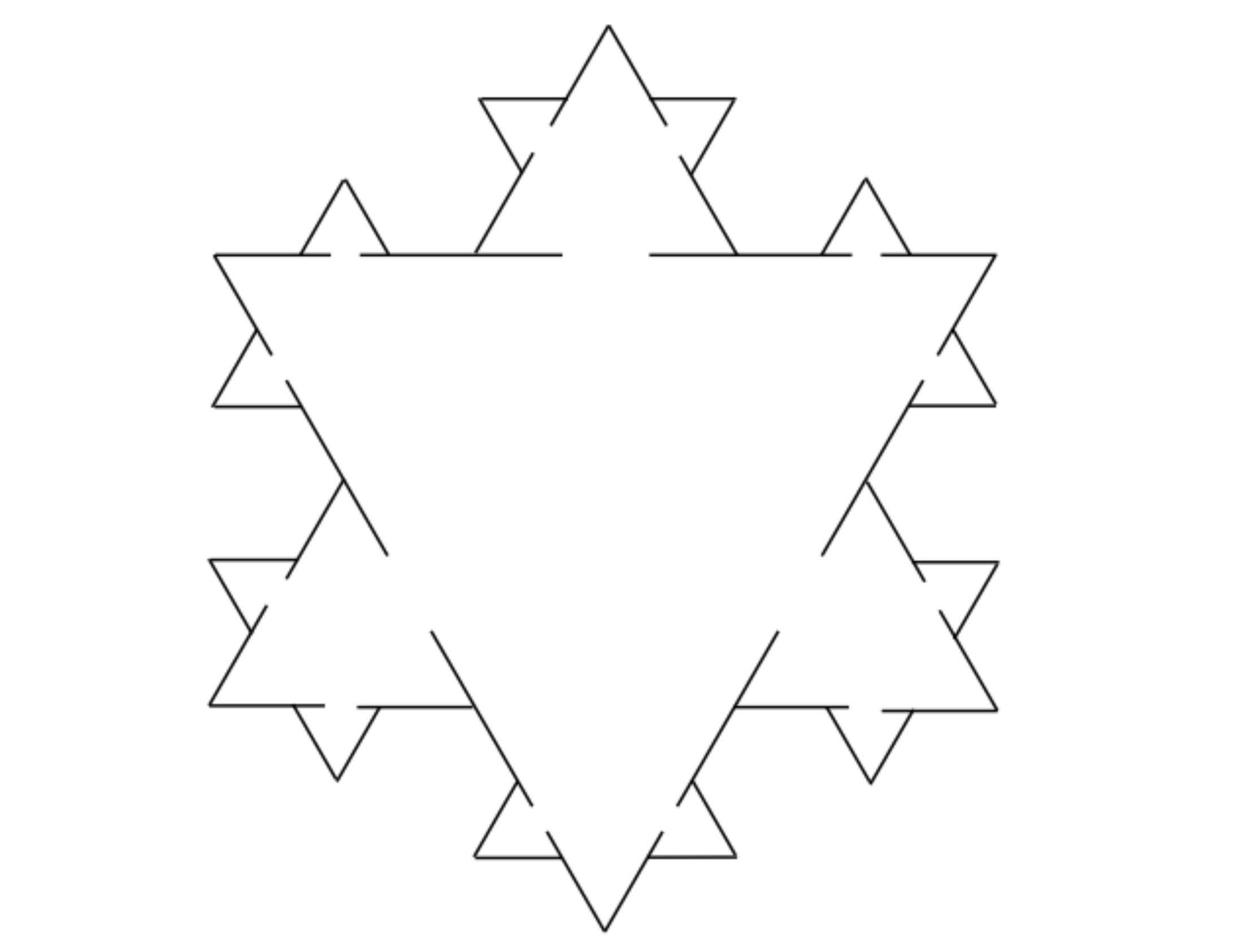} 
\caption{The modified Koch domain (pre-fractal, second step) introduced in \cite{BCM}. The passage between triangles is blocked by a wall with a small opening. The construction can be considered for random Koch domains. The time the process spends in a triangle strongly depends on the size of the opening. The asymptotic behavior of that size implies the construction of trap or non-trap Koch domains.}
\label{Fig:KochModified}
\end{figure}

\subsection{Non-local initial value problem on the RKD $\Omega^{(\boldsymbol{\xi})}$}

In \cite{CapDovJEE} we consider the sequence of time-changed process $X^{L,n} = X^n \circ L$ where $X^n$ is an elastic Brownian motion on $\Omega^{(\boldsymbol{\xi}|n)}$ with elastic coefficient $c_n$. We study the asymptotic behavior of $X^{L,n}$ depending on the asymptotics for $c_n$. In the deterministic case, the process $X^{L,n}$ can be associated with the non-local Cauchy problem on $\Omega^{(\xi|n)}$,
\begin{align*}
\mathfrak{D}^\Phi_t u = A^n u, \quad u_0=f \in D(A^n).
\end{align*}
Denote by $\mathbb{D}$ the set of continuous functions from $[0, \infty)$ to $E_\partial = \Omega^{(\xi)}\cup \partial$ which are right continuous on $[0, \infty)$ with left limits on $(0, \infty)$. Here, $\partial$ is the cemetery point, that is $E^n_\partial$ is the one-point compactification of $E^n=\Omega^{(\xi | n)}$, $n \in \mathbb{N}$. Denote by $\mathbb{D}_0$ the set of non-decreasing continuous function from $[0, \infty) $ to $[0, \infty)$. Let $X$ be the process on the random fractal $\Omega^{(\boldsymbol{\xi})}$ with generator $(A, D(A))$ where $A$ is the Neumann, the Dirichlet or the Robin Laplacian if $c_n \to c$ respectively with $c=0,\infty$ or a constant. 

\begin{theorem}
(\cite{CapDovJEE}) As $n \to \infty$,
\begin{align*}
X^{L, n} \to X^L \quad \textrm{in distribution in} \quad \mathbb{D} \quad \boldsymbol{\xi}-a.s. \quad \textrm{on} \quad \Omega^{(\boldsymbol{\xi})}.
\end{align*}
\end{theorem}
This result says that we have a time-change representation for the solution to the non-local Cauchy problem on  the (fractal) Koch snowflakes. 

The Koch domain (as the random Koch domain) is quite regular in a sense to be better explained. Under Definition \ref{def1}, Koch domains are non-trap for the Brownian motion and therefore we may say that the boundary $\partial \Omega^{(\boldsymbol{\xi})}$ is regular.

\section{Non-local boundary value problems}

\subsection{Singular kernels}

For a bounded subset $\Omega$ of $\mathbb{R}^d$ we first focus on the fractional boundary value problem 
\begin{align}
\label{FBVP}
\left\lbrace
\begin{array}{ll}
\displaystyle \frac{\partial u}{\partial t}(t,x) = \Delta u(t,x), \quad t>0, \; x \in \Omega \\
\\
\displaystyle \eta D^\alpha_t u(t,x) + \sigma \partial_{\bf n} u(t,x) + c\, u(t,x)=0, \quad t>0, \; x \in \partial \Omega\\
\\
\displaystyle u(0,x) = f(x), \quad x \in \overline{\Omega}
\end{array}
\right .
\end{align}
where $\eta, \sigma, c$ are positive constants, $\alpha \in (0,1)$, $\partial_{\bf n}u$ is the outer normal derivative and $D^\alpha_t u(t,x)$ is the Caputo-Dzherbashian fractional derivative. As before, we write $u_f= u_f(t,x)$ for the function such that 
\begin{align*}
u_f \to f \quad \textrm{as} \quad t \downarrow 0.
\end{align*}
Then, we introduce the space 
\begin{align*}
D_L = \left\lbrace \varphi = \varphi_{f|_{\partial \Omega}} :\, \Big| \frac{\partial \varphi}{\partial t}(t,x) \Big| \leq v(x) t^{\alpha-1}, \, v \in L^\infty (\partial \Omega) \right\rbrace.
\end{align*}
We also introduce the processes $\bar{X}=\{\bar{X}_t\}_{t\geq 0}$ and $\hat{X} = \{\hat{X}_t\}_{t\geq 0}$ defined as follows. Let $X=\{X_t\}_{t\geq 0}$ be a Brownian motion on $\overline{\Omega} = \Omega \cup \partial \Omega$ where $\Omega$ is a bounded subset of $\mathbb{R}^d$. Let $L=\{L_t\}_{t\geq 0}$ be an inverse to an  $\alpha$-stable subordinator $H=\{H_t\}_{t\geq 0}$. Define the time-changed process $\hat{X} = \{\hat{X}_t\}_{t\geq 0}$ as the composition $\hat{X}:= X \circ \hat{L}$ with $\hat{L}_t=t$ if $\hat{X}$ is running on $\Omega$ and $\hat{L}_t = L_t$ if $\hat{X}$ is on $\partial \Omega$. Furthermore, we define $\bar{X}$ as the composition of an elastic Brownian motion (Robin boundary condition for the associated Cauchy problem) and the inverse to the process
\begin{align*}
\bar{V}_t = t +  H \circ (\eta/\sigma) \gamma_t
\end{align*}
where $\gamma = \{\gamma_t\}_{t\geq 0}$ is here the local time on $\partial \Omega$ of the reflecting Brownian motion $X^+$ on $\overline{\Omega}$. We recall that an elastic Brownian motion can be written by considering the couple $(X^+, \gamma)$.\\

\begin{figure}
\includegraphics[scale=1]{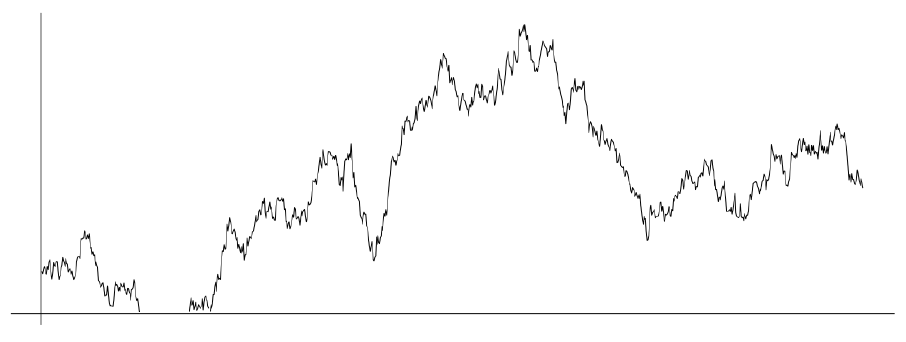} 
\caption{The $1$-dimensional path (a realization of $\hat{X}$ or $\bar{X}$) exhibits a trap effect at zero. The process stops for a random amount of time (according with $\bar{V}^{-1}$), then it reflects continuously. The same effect can be considered on the boundary $\partial \Omega$ for a $d$-dimensional motions with random time-change $\bar{V}^{-1}$.}
\end{figure}

We have the following result. 
\begin{theorem}
\label{thm:MAIN}
(\cite{DovFBVP2}) For the solution $u \in C^{1,2}((0, \infty), \overline{\Omega}) \cap D_L$ to the fractional boundary value problem \eqref{FBVP} we have that:
\begin{itemize}
\item[i)] The solution is given by
\begin{align}
\label{MAINrepu}
u(t,x) = \int_{\overline{\Omega}} f(y) \left( \int_0 p(s, x,y) l(t,s;y) ds\right) m(dy), \quad t >0, \, x \in \overline{\Omega} 
\end{align}
with the probabilistic representation
\begin{align}
\label{MAINrepXhat}
u(t,x) = \mathbf{E}_x[f(\hat{X}_t)], \quad t >0, \, x \in \overline{\Omega};
\end{align}
\item[ii)] Moreover, the following representation holds true
\begin{align}
\label{MINrepXbar}
u(t,x) 
= & \mathbf{E}_x\left[ f(\bar{X}_t) \right]	\notag \\
= & \mathbf{E}_x\left[f(X^+ \circ \bar{V}^{-1}_t) \exp \left( - (c/\sigma)\, \gamma \circ \bar{V}^{-1}_t \right) \right], \quad t>0,\; x \in \overline{\Omega}
\end{align}
where $\bar{V}^{-1}_t$ is the inverse to $\bar{V}_t = t + H \circ (\eta/\sigma) \gamma_t$. The process $\bar{X} = \{\bar{X}_t\}_{t\geq 0}$ can be constructed as a time-changed elastic Brownian motion.
\end{itemize}
\end{theorem}
The previous results gives a clear picture about the boundary behavior of the associated process near the boundary. In particular, by considering the representation in terms of $\bar{X}$, since the process $H$ may have jumps, then the inverse to $\bar{V}_t$ slows down the process on the boundary according with the plateaus associated with that jumps. As $\alpha =1$, the process $\bar{V}_t$ becomes $V_t = t + (\eta/\sigma) \gamma_t$ and $\bar{X}$ corresponds to the elastic sticky Brownian motion. \\

\emph{The fractional boundary value problem above introduces independent waiting times on the boundary $\partial \Omega$ in terms of $\bar{V}$. This means that a given stop can be explained by an endogenous effect. We exploit such a reading in order to provide a connection with trap boundaries.} \\

For example the process $\bar{X}$ on a (non-trap) Koch domain can be consdered in order to describe the motion of a Brownian motion $B^+$ on a (trap) Koch domain obtained as in Figure \ref{Fig:KochModified}.\\

In the papers \cite{DovFBVP1, DovFBVP2} the lifetime of the process $\bar{X}$ has been investigated together with an interesting connection with irregular domains. In particular, the time the process spends on the boundary can be related with a delayed reflection or an irregular boundary as well. The non-local boundary value problem (for a smooth domain) can be therefore associated with a local boundary value problem (for an irregular domain). \\

\begin{figure}
\includegraphics[scale=.3]{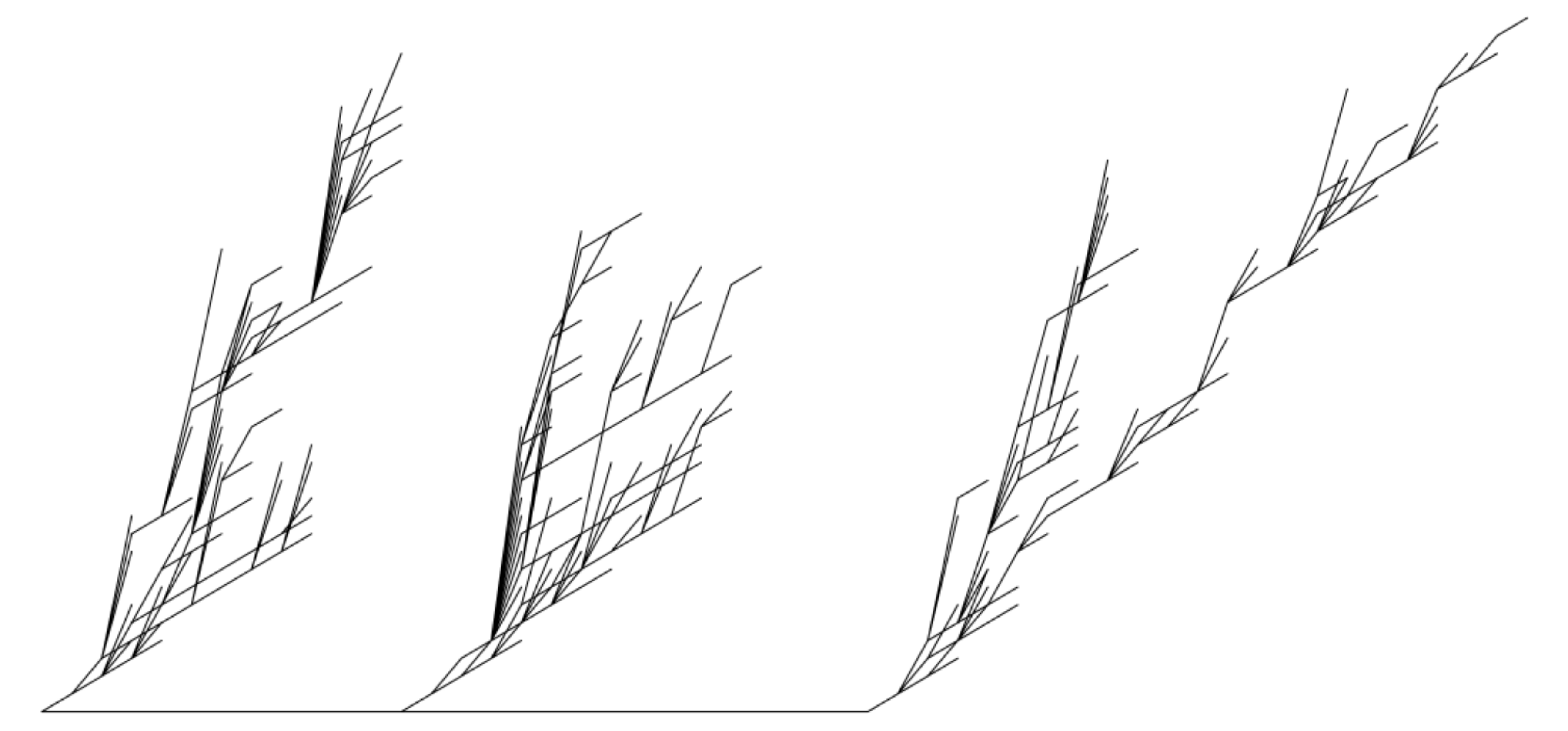} 
\caption{Graphs in a forest in which the distances are randomly determined by $\bar{V}^{-1}$. Each graph represents an excursion of $\bar{X}$ in the sense of Lukasiewicz path with a rotation for the edges of each vertex in order to obtain the (more attractive) Lichtenberg figures. The waiting time at zero gives the random distance between Lightning graphs.}
\end{figure}

A further result concerning non-local boundary conditions has been investigated in \cite{ColDov2}. We provide here only the main idea. Let $D_0 = D_1 \cup D_2$
be the set given by considering the sets
\begin{align*}
D_1 = \left\lbrace \varphi : \forall\, x>0,\; \varphi(\cdot, x) \in W^{1, \infty}(0, \infty) \;\; \textrm{and} \;\; \forall\, t>0,\;  \lim_{x \downarrow 0} \mathfrak{D}^\Psi_t \varphi(t,x) \; \textrm{exists} \right\rbrace
\end{align*}
and
\begin{align*}
D_2 = \left\lbrace \varphi : \forall\, t >0,\; \varphi(t, \cdot) \in W^{1,\infty}(0, \infty) \;\; \textrm{and} \;\; \forall\, t>0,\; \lim_{x\downarrow 0} \mathbf{D}^\Phi_{x+} \varphi(t,x) \; \textrm{exists} \right\rbrace .
\end{align*}
Then, we state the following result concerning the problem
\begin{equation}
\label{ProblemColDov}
\left\lbrace
\begin{array}{ll}
\displaystyle \frac{\partial u}{\partial t}(t,x) = \frac{\partial^2 u}{\partial x^2} (t,x), & t>0, \; x \in (0, \infty)\\
\\
\displaystyle \eta \mathfrak{D}^\Psi_t u(t,0) + \sigma \mathbf{D}^\Phi_{x+} u(t,0)=0, & t>0\\
\\
\displaystyle u(0, x) = f(x), & x \in [0, \infty)
\end{array}
\right .
\end{equation}
for an initial datum $f \in C[0, \infty) \cap L^\infty(0,\infty)$ and positive constants $\eta, \sigma$.

\begin{theorem}
(\cite{ColDov2}) For the solution $u \in C^{1,2}((0, \infty), (0,\infty)) \cap D_0$ to the problem \eqref{ProblemColDov} we have the probabilistic representation
\begin{align*}
u(t,x) = \mathbf{E}_x[ X^\bullet_t] = \mathbf{E}_x[f(X^+ \circ \bar{V}^{-1}_t + R_{\gamma} \circ \bar{V}^{-1}_t)], \quad t>0,\; x \in [0, \infty)
\end{align*}
where 
\begin{itemize}
\item[-] the couple $(X^+, \gamma)$ has been defined above;
\item[-] $\bar{V}^{-1}_t$ is an inverse as defined in the previous theorem with $H=H^\Psi$;
\item[-] $R_t= H \circ L_t - t$ is the composition of a subordinator and its inverse with $H=H^\Phi$.
\end{itemize}
\end{theorem}
A detailed discussion on the paths behavior of that process is given in \cite{ColDov2}, the process near the boundary jumps inside the domain because of $R_t$ and then it stops there according with $\bar{V}^{-1}_t$. Here we recall that a subordinator can be obtained from a compound Poisson process (see \cite{DovCRW}) written as
\begin{align*}
\frac{1}{n} \sum_{k=1}^{N_{nt}} \epsilon_k e^{\chi_k}, \quad t>0,\quad n \in \mathbb{N}
\end{align*} 
where $N_t$ is a Poisson process, $\chi_k$ are exponential random variables such that $\mathbf{P}(\chi_k > x) = e^{-\alpha x}$ and $\mathbf{P}(\epsilon_k = 1)=1$. The case in which $\mathbf{P}(\epsilon_k = +1)=p$ and $\mathbf{P}(\epsilon_k = -1)=q$ leads to the limit $H_{pt} - H_{qt}$ as $n \to \infty$. Concerning our case, we can approximate the last jump of $H$ given by $R$ by considering the jump $Y_k=\epsilon_k e^{\chi_k}$ with power law $P(Y_k \in dy) = \alpha n^{-\alpha} y^{-\alpha-1} \mathbf{1}_{(y > 1/n)} dy$. The jump never occur if its length is less than a given threshold.\\

In case $\mathbf{D}^\Phi_{x+}$ is written in terms of a non-singular kernel, the jumps can be well-described as in the next section, the process $R_t$ can be associated with a compound Poisson process. This boundary behavior with jumps has been first investigated in \cite{IM63, IM74} with no mention about non-local operators which have been considered in \cite{ColDov2}. The novelty concerned with random holding time at the boundary has been first treated in \cite{DovFBVP1, DovFBVP2} by extending the sticky boundary condition introduced by Feller in \cite{Feller52} and Wentzell in \cite{Wentzell59} with the probabilistic representation in terms of sample paths obtained by It\^{o} and McKean \cite[Section 10]{IM74}.

\begin{figure}
\includegraphics[scale=.9]{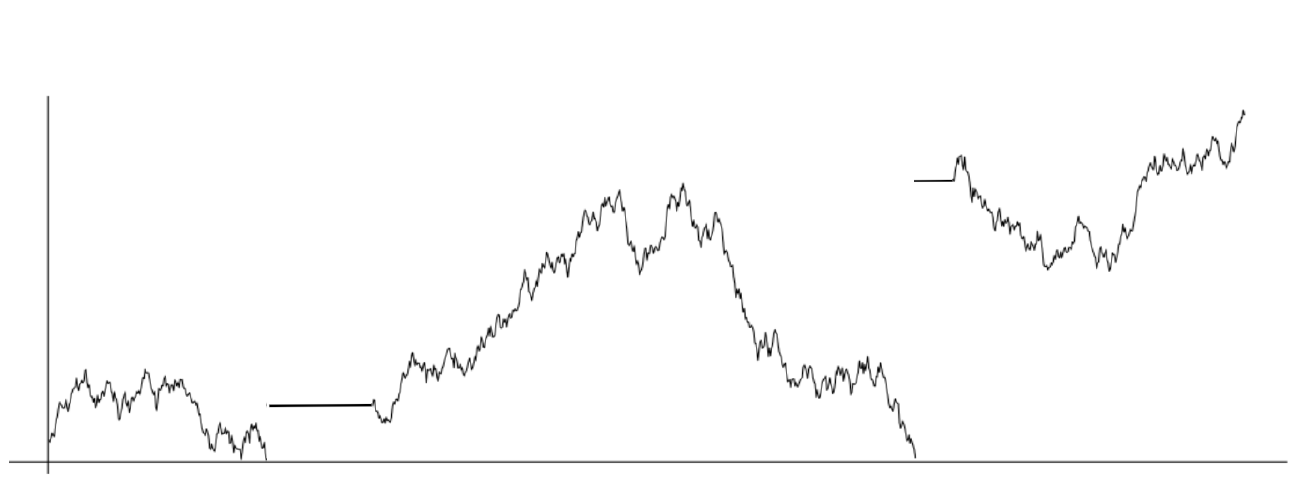} 
\caption{The path of $X^+ \circ \bar{V}^{-1}_t + R_{\gamma} \circ \bar{V}^{-1}_t$. The process exhibits a {\it jump and stop behavior} near the boundary. Indeed, the process never hits the zero point boundary, it jumps near the boundary according with $R_\gamma$ and then it stops for a random amount of time according with $\bar{V}^{-1}$.}
\end{figure}

The connection between local and non-local boundary value problem seems to be quite intriguing. Consider the problems
\begin{equation}
\left\lbrace
\begin{array}{ll}
\displaystyle \frac{\partial u}{\partial t}(t,x) = \mu \frac{\partial u}{\partial x}(t,x) + \frac{\partial^2 u}{\partial x^2}(t,x), & t>0,\; x\in (0, \infty) \\
\\
\displaystyle \frac{\partial u}{\partial x}(t,0) =c_0 u(t,0), & t>0\\
\\
\displaystyle u(0, x) = f(x), & x \in [0, \infty)
\end{array}
\right .
\end{equation}
and
\begin{equation}
\left\lbrace
\begin{array}{ll}
\displaystyle \frac{\partial u}{\partial t}(t,x) = \mu \frac{\partial u}{\partial x}(t,x) + \frac{\partial^2 u}{\partial x^2}(t,x), & t>0,\; x\in (0, \infty) \\
\\
\displaystyle \mathfrak{D}^{\Phi}_t u(t,0) + c_1 u(t,0) = c_2, & t>0\\
\\
\displaystyle u(0, x) = f(x), & x \in [0, \infty)
\end{array}
\right .
\end{equation}
where $\Phi(\lambda) = \sqrt{\lambda + (\mu/2)^2} - \sqrt{(\mu/2)^2}$ introduces the so-called tempered derivative (of Caputo-Dzherbashian type). For the sake of simplicity, here we consider $c_0$, $c_1$, $c_2$ as given constants. In \cite{DovIaf1}, starting from a previous result in \cite{DovIafOrs1}, it has been shown that the solution coincides in case $f$ is constant. Thus, for the elastic drifted Brownian motion, say $X^{\mu, c_0}_t$, we can write
\begin{align*}
u(t,x) = \mathbf{E}_x[f(X^{\mu, c_0}_t) M^{\mu, c_0}_t]
\end{align*} 
where the multiplicative functional associated with the elastic condition can be written in terms of $L$, the inverse to a subordinator $H$ with symbol $\Phi$. In particular, $M^{\mu, c_0}$ is equivalent to $\bar{M}^{c_1, c_2} = \bar{M}^{c_1, c_2}(L)$ which is associated with the non-local (tempered) boundary condition.\\

In the literature some works deal with non-local boundary value problems with different meaning (see for example \cite{AsBe, ElS, NtTs}). Our approach is motivated by the clear effort to find a description for anomalous behavior near smooth boundaries with possible application in case of irregular boundaries.

\subsection{Non-singular kernels}
We focus on a famous operator which has been deeply discussed in the last few years. Recently, in a series of papers starting from \cite{CapFab}, some authors considered the Caputo-Fabrizio operator (for a given constant $a$)
\begin{align*}
\mathscr{D}^{(\alpha)}_x u(x) = \frac{ M(\alpha)}{1-\alpha} \int_{a}^x u^\prime(\tau) \exp\left[ - \frac{\alpha\, (x- \tau)}{1-\alpha} \right]d\tau
\end{align*}
where $M(\alpha)$ is a normalizing function such that $M(0)=M(1)=1$ and $\alpha \in (0,1)$ is the fractional order of the operator. The operator has been introduced as a fractional derivative although the proper definition of fractional operators is now stimulating a deep debate (\cite{OrtMac}). However, the operator $\mathscr{D}^{(\alpha)}_x$ can be regarded as a non-singular operator and can be included in the class of integro-differential operators. As pointed out in \cite{CapFab}, we have that $\mathscr{D}^{(\alpha)}_x u(x) \to u^\prime(x)$ as $\alpha \to 1$ and $\mathscr{D}^{(\alpha)}_t u(x) \to u(x) - u(a)$ as $\alpha \to 0$. For the sake of simplicity we assume here that $a=0$ and $u(0)=0$. Our results are concerned with the boundary value problem in which the boundary-operator
\begin{align}
\mathscr{D}^{(\alpha)}_x u  + c \alpha\, u=0, \qquad c\geq 0
\label{BO-1}
\end{align}
is considered. As a consequence of the properties illustrated before for the Caputo-Fabrizio operator, as $\alpha\to 0$ or $\alpha\to 1$ we respectively have 
\begin{align}
u(x)=0 \qquad \textrm{or} \qquad u^\prime(x) = - c\,u(x)
\label{cond-1}
\end{align}
in place of \eqref{BO-1}. The conditions \eqref{cond-1} respectively correspond to Dirichlet or Robin condition. For $\alpha \in [0,1]$ and $c \geq 0$, we consider the solution to
\begin{align*}
\frac{\partial u}{\partial t} = G u, \quad u_0 = f
\end{align*}
with
\begin{align*}
(\mathscr{D}^{(\alpha)}_x u + c \alpha \, u ) |_{x=0} = 0
\end{align*}
where $G$ is the generator of a Markov process on $[0, \infty)$.

In \cite{CapFab}, the authors provided some properties of $\mathscr{D}^{(\alpha)}_x u(x)$. For $a=0$ for instance, the Laplace transform reads as follows (as usual we denote by $\widetilde{u}$ the Laplace transform of $u$)
\begin{align}
\int_0^\infty e^{-\lambda x} \mathscr{D}^{(\alpha)}_x u(x)\, dx = & \frac{M(\alpha)}{1-\alpha} \left( \lambda \widetilde{u}(\lambda) - u(0) \right) \int_0^\infty e^{-\lambda x} e^{- \frac{\alpha\, x}{1-\alpha}} dx \notag \\
= & M(\alpha) \frac{1}{\alpha + \lambda (1-\alpha)} \left( \lambda \widetilde{u}(\lambda) - u(0) \right) \label{CFop-1}.
\end{align}
From now on, we assume that 
\begin{align}
M(\alpha)=1 \quad \textrm{and} \quad a=0.
\end{align}
The Laplace transform \eqref{CFop-1} corresponds to the Laplace transform of the fractional operator $\mathfrak{D}^\Phi_x$ with L\`evy symbol and measure
\begin{align}
\Phi(\lambda) = (\theta + 1) \frac{\lambda }{\theta + \lambda} \quad \textrm{and} \quad \bar{\phi}((y, \infty)) = (\theta +1) e^{- \theta y}, \quad \theta = \frac{\alpha}{1-\alpha}
\label{Phi-CF}
\end{align}
where, as a quick check shows, 
\begin{align*}
\lim_{\lambda \downarrow 0} \frac{\Phi(\lambda)}{\lambda} = \frac{1}{\alpha}. 
\end{align*}
From a probability view-point, the symbol $\Phi$ characterizes a subordinator $H$, that is a non negative, non-decreasing process. Notice that, $\bar{\phi}((0, \infty)) < \infty$ if $\alpha < 1$ and the process $H_t$ is a step-process, its trajectories are not strictly increasing. The inverse to a subordinator can be defined as 
\begin{align*}
L_t = \inf\{s\geq 0\,:\, H_s >t\}, \quad t>0
\end{align*}
from which we also obtain the relation
\begin{align}
\mathbf{P}_0(H_t < x) = \mathbf{P}_0(L_x > t), \quad x,t>0.
\label{rel-P-LH}
\end{align}
As $\alpha\to 1$ we obtain $\Phi(\lambda)= \lambda$ as expected, the process $H$ corresponds to the elementary subordinator $H_t=t$. We recall that $L_t$ may have plateaus only if the process $H_t$ may have jumps (at least for $\alpha\neq 1$). Since $H_t$ has non-decreasing paths, the inverse $L_t$ may have jumps ($L$ is not continuous).  Thus, $H$ is not strictly increasing and $L_t$ can not be continuous. As $\alpha\to 0$, the ratio $\Phi(\lambda)/\lambda$ gives an infinite mean life time and therefore a delayed process on the boundary. This actually corresponds to an absorption at $x=0$ whereas the Dirichlet condition suggests a kill. We still have (\cite{CapDovFCAA})
\begin{align}
\mathbf{E}_0 \left[ \int_0^\infty e^{-\lambda t} f(L_t) dt  \right] = \frac{\Phi(\lambda)}{\lambda} \mathbf{E}_0 \left[ \int_0^\infty e^{-\lambda H_t} f(t) dt \right].
\end{align}
Thus, for $f=\mathbf{1}_D$ with $D \subseteq (0, \infty)$,
\begin{align}
\int_0^\infty e^{-\lambda t} \mathbf{P}_0(L_t \in D) dt = \frac{\Phi(\lambda)}{\lambda} \int_D e^{- t \Phi(\lambda)} dt.
\end{align}
We also write
\begin{align*}
\int_0^\infty e^{-\lambda t} \mathbf{P}_0(L_t \in dx) dt = \frac{\Phi(\lambda)}{\lambda} e^{-x \Phi(\lambda)} dx 
\end{align*}
and the density $l$ of $L$ solves 
\begin{align*}
\mathscr{D}^{(\alpha)}_t l = - \partial_x l
\end{align*}
where $L$ is an inverse of $H$ with symbol $\Phi(\lambda) = (\theta +1) \frac{\lambda}{\theta + \lambda}$. It can be easily shown that $H$ is a compound Poisson process. Indeed, consider the process 
\begin{align*}
H_t=\sum_{k=1}^{N_t} Y_k
\end{align*}
where $N_t\sim Pois((\theta+1)t)$ and $Y_k$ are i.i.d. random variables with $Y_1\sim Exp(\theta)$, $\theta>0$. For the reader's convenience, we recall that, for $N_t \perp Y_1$,
\begin{align*}
\mathbf{E}[e^{-\lambda H_t}] 
= & \mathbf{E}\left[ \prod_{k=1}^{N_t} e^{-\lambda Y_1}  \right] = \mathbf{E} \left[ \left( \mathbf{E}[e^{-\lambda Y_1}] \right)^{N_t} \right]\\ 
= & \exp \left( - (\theta + 1) t \left(1 - \mathbf{E}[e^{-\lambda Y_1}] \right) \right).
\end{align*}
Thus, it turns out that
\begin{align*}
\mathbf{E}[e^{-\lambda H_t}] = \exp \left(-t (\theta+1) \int_0^\infty (1-e^{-\lambda y}) \theta e^{-\theta y} dy \right), \quad \lambda>0.
\end{align*}
Our focus is on the boundary condition
\begin{align*}
\mathscr{D}^{(\alpha)}_x u(t,0) + c\alpha u(t,0) = 0
\end{align*}
corresponding to the drifted process $H_t + c\alpha t =: \bar{H}_t$. The process $\bar{H}_t$ may have jumps, never plateaus except for $c=0$. This process is therefore strictly increasing with continuous inverse $\bar{H}_t^{-1}$ for which (see also \cite{BegDov1})
\begin{align*}
\mathbf{P}(\bar{H}^{-1}_t > s) = \mathbf{P}(t < \bar{H}_s) = \mathbf{P}(t-c\alpha s < H_s) 
\end{align*}
and
\begin{align*}
\int_0^\infty e^{-\lambda t}  \mathbf{P}(\bar{H}^{-1}_t > s) dt = \frac{1}{\lambda} e^{-s (c\alpha \lambda + (\theta+1) \frac{\lambda}{\theta +\lambda})} , \quad \lambda>0.
\end{align*}
By considering the symbol $\Phi(\lambda)= c\alpha \lambda + \frac{\theta+1}{\theta+\lambda} \lambda$ we have the (space) non-local boundary condition discussed in the previous section. The interested reader can also consult \cite{BegCap} for a discussion on this operator.

Many other cases can be considered. Indeed, the operator
\begin{align}
\label{straightforwardC} 
\mathscr{D}^{\Phi}_x u(x) := \int_0^x u^\prime \mathscr{K}(x-\tau) d\tau
\end{align} 
with the non-singular kernel $\mathscr{K}$ can be associated with a compound Poisson process as above. Let us consider $\mathbf{P}(Y_k \in dy) = f_{Y_1}(y)dy$ for the jumps from the boundary where $Y_k$ takes values in some set $supp(Y_k)=supp(f_{Y_1}) \subset (0,\infty)$. We do not consider negative jumps. Then, we get
\begin{align*}
\mathbf{E}[e^{-\lambda H_t}] = \exp \left(- t (\theta + 1) \int_0^\infty \big( 1 - e^{-\lambda y} \big) f_{Y_1}(y)dy \right), \quad \lambda>0
\end{align*}
and, for $u$ such that $u(x)=0$ as $x\leq 0$, 
\begin{align*}
\mathscr{D}^{\Phi}_x u(x) 
= & (\theta + 1) \int_0^\infty \big( u(x) - u(x-y) \big) f_{Y_1}(y)dy\\
= & - (\theta + 1) \int_0^\infty f_{Y_1}(y) \int_0^y \frac{du}{dz} (x-z) dz dy\\
= & - (\theta +1) \int_0^\infty \frac{du}{dz} (x-z) \int_z^\infty f_{Y_1}(y) dy\, dz\\
= & (\theta +1) \int_0^x u^\prime(x-z) \, \mathbf{P}(Y_1 > z) \, dz.
\end{align*}
Thus, straightforward calculations lead to \eqref{straightforwardC} with non-singular kernel
\begin{align*}
\mathscr{K}(z) = (\theta +1) \, \mathbf{P}(Y_1 > z), \quad \theta >-1.
\end{align*}
A quick example is given by $Y_1$ distributed as a Mittag-Leffler random variable for which $\mathbf{P}(Y_1 >y) = E_\alpha (-r y^\alpha)$ with $r>0$ and $\alpha\in (0,1]$. It is well-known that for $\alpha=1$ we have $Y_1 \sim Exp(r)$ obtaining the case discussed above.

\vspace{1cm}

{\bf Acknowledgements.} 
The author thanks his institution Sapienza University of Rome and the group INdAM-GNAMPA for the support under their Grants. The authors also thank the anonymous reviewers for their careful reading of the manuscript and their suggestions.

\end{document}